\documentclass[12pt,fleqn]{article}

\usepackage{todonotes}
{}

\usepackage[french, english]{babel}

\usepackage[utf8]{inputenc}
\usepackage[T1]{fontenc}

\usepackage{mathtools}
\usepackage{amsfonts,amsmath, amsthm,amssymb,esint}
\usepackage{bbm}

\allowdisplaybreaks

\usepackage{Alegreya}
\usepackage{float}
\usepackage{authblk}
\usepackage{tikz}
\usepackage{color}
\definecolor{DB}{rgb}{0.0,0.0,0.8} 
\definecolor{DG}{rgb}{0.0,0.55,0.14}
\definecolor{DR}{rgb}{0.5,0,0.07}

\usepackage[hyperindex=true,colorlinks=true, urlcolor={DB}, linkcolor={DB}, menucolor={DB}, citecolor={DG}, anchorcolor={DG},linktoc=all]{hyperref}

\usepackage{url}

\usepackage{amsrefs}

\usepackage{bbding}
\newcommand{\verti}[1]{{\left\vert #1 \right\vert}}
\newcommand{\vertii}[1]{{\left\vert\kern-0.25ex\left\vert #1 \right\vert\kern-0.25ex\right\vert}}  
\newcommand{\vertiii}[1]{{\left\vert\kern-0.25ex\left\vert\kern-0.25ex\left\vert #1 \right\vert\kern-0.25ex\right\vert\kern-0.25ex\right\vert}}
\newcommand{\pari}[1]{{\left( #1 \right)}}

\def\bt{\begin{theo}}
\def\et{\end{theo}}
\def\bpr{\begin{prop}}
\def\epr{\end{prop}}
\def\bl{\begin{lemma}}
\def\el{\end{lemma}}
\def\bc{\begin{coro}}
\def\ec{\end{coro}}
\def\br{\begin{rema}}
\def\er{\end{rema}}
\def\bp{\begin{proof}}
\def\ep{\end{proof}}

\numberwithin{equation}{section}

\usepackage{mathrsfs}
\usepackage{stmaryrd}

\usepackage{graphicx}

\usepackage[margin=25mm]{geometry}

\def\R{{\mathbb R}}
\def\N{{\mathbb N}}

\def\fo{\forall\,}

\def\bes{\begin{equation*}}
\def\ees{\end{equation*}}
\def\be{\begin{equation}}
\def\ee{\end{equation}}
\def\ba{\begin{aligned}}
\def\ea{\end{aligned}}
\def\d{\displaystyle}
\def\na{\nabla}

\usepackage{enumerate}
\def\ben{\begin{enumerate}}
\def\een{\end{enumerate}}

\theoremstyle{definition}
\newtheorem{prop}{Proposition}[section]
\newtheorem{theo}[prop]{Theorem}
\newtheorem{coro}[prop]{Corollary}
\newtheorem{lemma}[prop]{Lemma}
\theoremstyle{definition}

\theoremstyle{definition}
\newtheorem{rema}[prop]{Remark}
\theoremstyle{definition}

\def\ep{\end{proof}}
\def\bp{\begin{proof}}

\newcommand\blfootnote[1]{%
  \begingroup
  \renewcommand\thefootnote{}\footnote{#1}%
  \addtocounter{footnote}{-1}%
  \endgroup
}

\usepackage[autostyle=true]{csquotes}
\title{Higher-order affine Sobolev inequalities}
\author[Tristan]{Tristan Bullion-Gauthier*}
\date{}

\begin{document}
\maketitle

\begin{abstract}
Zhang refined the classical Sobolev inequality $\|f\|_{L^{Np/(N-p)}} \lesssim \| \nabla f \|_{L^p}$, where $1\leq p <N$, by replacing $\|\nabla f\|_{L^p}$ with a smaller quantity invariant by unimodular affine transformations. The analogue result in homogeneous fractional Sobolev spaces $\mathring{W}^{s,p}$, with $0 <s <1$ and $sp < N$, was obtained by Haddad and Ludwig. We generalize their results to the case where $s > 1$. Our approach, based on the existence of optimal unimodular transformations, allows us to obtain various affine inequalities, such as affine Sobolev inequalities, reverse affine inequalities, and affine Gagliardo--Nirenberg type inequalities. In a different but related direction, we also answer a question concerning reverse affine inequalities, raised by Haddad, Jim\'enez, and Montenegro.     
\end{abstract}
\blfootnote{*Universite Claude Bernard Lyon 1, CNRS, Centrale Lyon, INSA Lyon, Université Jean Monnet, ICJ UMR5208, 69622 Villeurbanne, France.  
\url{bullion@math.univ-lyon1.fr}}
\blfootnote{Acknowledgments. I thank \'Oscar Dom\'inguez for introducing me to affine inequalities and for his wonderful questions, which are at the heart of this article. It is largely the result of writing advice from Petru Mironescu, to whom I am indebted for taking the time to teach me how to do and write mathematics. I am grateful to Antoine Detaille for his careful reading and for his ideas to improve the presentation.}

\section{Introduction}
The classical Sobolev inequality asserts that,  for each $1 \leq p<N$, there exists $\widetilde{C}_{p,N}<\infty$ such that
\begin{flalign}\label{classic SOb}
    \vertii{f}_{L^{Np/(N-p)}(\R^N)} \leq \widetilde C_{p,N} \left( \int_{\R^N} \verti{\nabla f(x)}^p \ dx \right)^{1/p} \text{,} \ \fo f \in \mathring{W}^{1,p}(\R^N)\text{.}
\end{flalign}
(Here, $\mathring{W}^{1,p}(\R^N)\coloneq\{ f\in L^{Np/(N-p)}(\R^N);\, \na f\in L^p(\R^N)\}$\text{,} which coincides with the completion of the space of compactly supported smooth functions with respect to the semi-norm $\vertii{\nabla \cdot }_{L^p}$.)
The sharp value of the constant $\widetilde C_{p,N}$ was found by Aubin \cite{aubin1976problemes} and Talenti \cite{talenti1976best}.
  In his seminal article \cite{zhang1999affine}, Zhang improved the sharp Sobolev inequality  in the case where $p=1$ and proved the  \enquote{affine Sobolev inequality}
  \be\label{zhang}
  \begin{split}
      \vertii{f}_{L^{N/(N-1)}(\R^N)} \leq  C_{1,N} \left(\int_{\mathbb{S}^{N-1}}\left(\int_{\R^N}\verti{\nabla f(x) \cdot \xi} \ dx\right)^{-N} \, d\mathscr{H}^{N-1}(\xi) \right)^{-1/N}, 
      \\
      \fo f \in \mathring{W}^{1,1}(\R^N).  
      \end{split}
  \ee
  
  Here, the constant $ C_{1,N}$ is such that, when $f$ is radial, the right-hand sides of \eqref{classic SOb} and \eqref{zhang} coincide. By a straightforward application of Jensen's inequality, one finds that the right-hand side of \eqref{zhang} is less than or equal to the one of  \eqref{classic SOb}, and thus \eqref{zhang} is a refinement of \eqref{classic SOb}. An important feature of \eqref{zhang}  is its invariance under unimodular linear transformations (i.e. $T \in \text{GL}_N$ such that $\verti{\text{det} \ T}=1$). This underlying property is characteristic of  the affine inequalities in the spirit of  \eqref{zhang}.

\smallskip
  The work of Zhang inspired many subsequent developments. In particular, Lutwak, Yang, and Zhang proved sharp affine Sobolev inequalities in the whole range $1 \leq p<N$ \cite{lutwak2002sharp}, while Wang proved an affine Sobolev inequality for $\text{BV}(\R^N)$ functions \cite{wang2012affine}. In a slightly different, but related direction, Cianchi, Lutwak, Yang, and Zhang \cite{cianchi2009affine} proposed a unified approach to such inequalities going beyond the critical value $p=N$.  

  \smallskip
  More recently,  Haddad and Ludwig  established a fractional counterpart of \eqref{zhang} \cites{haddad2024affine,haddad2022affine}. More precisely, these authors proved that, for $0<s<1$ and $1\leq p<\infty$ verifying $sp<N$, and for each $f \in \mathring{W}^{s,p}(\R^N)$,  we have
  \be\label{lud had}
  \begin{split}
      &\vertii{f}_{L^{Np/(N-sp)}(\R^N)} \\
      &\le C_{s,p,N} \left(\int_{\mathbb{S}^{N-1}}\left(\int_{0}^{\infty}t^{-sp-1}\vertii{\Delta_{t\xi} f}_{L^p(\R^N)}^p \ dt \right)^{-N/sp} \, d\mathscr{H}^{N-1} (\xi)\right)^{-s/N}, 
      \end{split}
  \ee
  where the best constant $C_{s,p,N}$ is given by an explicit formula involving a best Sobolev constant $\widetilde C_{s,p,N}$ (similarly to above). 
   (Here, $\Delta_hf(x)\coloneq f(x+h)-f(x)$.) Their sharp result implies, by extrapolation ($s \rightarrow 1^{-}$), \eqref{zhang} and its extension to $\mathring W^{1,p}$. In a related direction, a new approach to affine Moser-Trudinger inequalities was proposed in \cite{dominguez2024new}.

 \smallskip
 We now present our contributions. 
\textit{The main goal of this article is to obtain affine Sobolev inequalities of general smoothness order $s$ (not necessarily $\leq 1$).} Since, when $N=1$, affine Sobolev inequalities coincide with standard Sobolev inequalities, \textit{in what follows we always assume that $N\ge 2$, unless otherwise stated}. 

\smallskip
Given $f\in \dot{W}^{s,p}(\R^N)$ (for the definition of $\dot{W}^{s,p}(\R^N)$, see \eqref{def 1} and \eqref{def 2}), we denote
\bes
 \mathscr{E}_{s,p}(f)\coloneq \sigma_N^{(N+sp)/Np}\left(\int_{\mathbb{S}^{N-1}} \left(\int_{0}^{\infty}t^{-sp-1} \vertii{\Delta^{\lfloor s \rfloor +1}_{t\xi}f}_{L^p(\R^N)}^p \, dt\right)^{-N/sp}\,  d\mathscr{H}^{N-1}(\xi) \right)^{-s/N},
\ees
 if $s>0$ is not an integer, respectively
\bes
    \mathscr{E}_{s,p}(f)\coloneq \sigma_N^{(N+sp)/Np} \left(\int_{\mathbb{S}^{N-1}} \left(\int_{\R^N}\verti{\partial^s_{\xi} f(x)}^p \ dx\right)^{-N/sp} d \mathscr{H}^{N-1}(\xi)\right)^{-s/N}, 
    \ees
if $s\ge 1$ is an integer. Here, $\sigma_N$ is the surface area of the unit sphere $\mathbb{S}^{N-1}$.

 \smallskip

Our first main results are the following.
\begin{theo}\label{subcrit aff emb}
    Let $s>0$ and $1\leq p<\infty$ satisfy $sp<N$. Then there exists $K=K_{s,p,N}<\infty$ such that
    \begin{flalign}\label{h order aff sob}
        \vertii{f}_{L^{Np/(N-sp)}(\R^N)} \leq K\mathscr{E}_{s,p}(f), \ \forall f \in \mathring{W}^{s,p}(\R^N),
    \end{flalign}
    possibly except when $s\geq 2$ is an integer and $p=1$.
\end{theo}
\label{thm1.2}

\begin{theo}\label{gen embedding}
    Let $ 0 < s_1< s_2$ and $1 \leq p_1,p_2 <\infty$ satisfy
    \begin{flalign}\label{Sob n}
        s_2 -\frac{N}{p_2}=s_1-\frac{N}{p_1}.
    \end{flalign}
    Then there exists $K=K_{s_1,p_1,s_2,p_2,N}<\infty$ such that
    \begin{flalign}\label{multi cases}
        \mathscr{E}_{s_1,p_1}(f) \leq K\mathscr{E}_{s_2,p_2}(f), \ \fo f \in \dot{W}^{s_1,p_1}(\R^N) \cap \dot{W}^{s_2,p_2}(\R^N),
  \end{flalign}
   possibly except when $s_2\geq 2$ is an integer and $p_2=1$.
\end{theo}
We emphasize the fact that our approach is new, even in the known case where $0<s\leq 1$.
One of its features is that, while it encompasses the case $0<s\leq  1$, it does not provide the sharp constants in \eqref{lud had}. The trade-off is that we gain in generality, but lose in precision. This pertains to the fact that
the sharp constants in \eqref{zhang} and \eqref{lud had} are obtained using rearrangements and convex geometry techniques which do not seem to have counterparts for higher-order inequalities.

\smallskip
The starting point of our proofs of Theorems \ref{subcrit aff emb} and \ref{gen embedding} is inspired by the results of  Huang and Li \cite{huang2016optimal}, who proved the following. 
\begin{enumerate}
    \item For each $f \in W^{1,p}(\R^N)$, there exists $T_f \in \text{SL}_N$ such that 
    \begin{flalign}\label{huang 1}
        \vertii{\nabla(f \circ T_f) }_{L^p(\R^N)}= \min\{\vertii{\nabla(f\circ T)}_{L^p(\R^N)};\, T \in \text{SL}_N \}. 
    \end{flalign}
    \item There exists $C<\infty$ such that if $f \in W^{1,p}(\R^N)$ satisfies 
    \begin{flalign*}
        \vertii{\nabla f}_{L^p(\R^N)}= \min \{\vertii{\nabla(f\circ T)}_{L^p(\R^N)};\, T \in \text{SL}_N \},
        \end{flalign*}
        then
    \begin{flalign}\label{huang 2}
        \vertii{\nabla f}_{L^p(\R^N)} \leq C \vertii{\nabla f \cdot \xi}_{L^p(\R^N)}, \ \fo \xi \in \mathbb{S}^{N-1}.
    \end{flalign}
\end{enumerate}  

In other words, for each $f \in W^{1,p}(\R^N)$, one can choose a representative of $f$ in the class $[f]_{1,p}\coloneq \{ f\circ T;\, T \in \text{SL}_N \}$ which has large directional derivatives
in all directions. For this representative, the $W^{1,p}$-analogue of \eqref{zhang} (possibly not with sharp constants) is equivalent to the Sobolev embedding \eqref{classic SOb}.

\smallskip
A striking conclusion of our analysis is  that the general affine Sobolev inequalities \eqref{h order aff sob} and \eqref{multi cases} are  equivalent to their classical counterparts, if we disregard the matter of finding the best constants. This follows from Theorems \ref{there is min intro} and \ref{constraint on min intro} below.

\begin{theo}\label{there is min intro}
    Let $s>0$ and $1 \leq p < \infty$.
    For each $f \in \dot{W}^{s,p}(\R^N)$, there exists $T_f \in \text{SL}_N$ such that
    \begin{flalign*}
        \verti{f \circ T_f}_{W^{s,p}}= \min\{\verti{f \circ T}_{W^{s,p}};\,  \ T \in \text{SL}_N \}.
    \end{flalign*}
\end{theo}

A companion of this theorem is the following counterpart of \eqref{huang 2}.
\begin{theo}\label{constraint on min intro} \hfill
\begin{enumerate}[(1)]
  \item For every non-integer $s$ and $1 \leq p < \infty$,  there exist $0<C^1_{s,p,N}\leq C^2_{s,p,N}<\infty$
  such that if $f \in \dot{W}^{s,p}(\R^N)$  satisfies
  \begin{flalign*}
    \verti{f}_{W^{s,p}}= \min \{\verti{f \circ T}_{W^{s,p}};\, T \in \textrm{SL}_N \},
  \end{flalign*}  
  then
  \begin{flalign}\begin{split}
    C^1_{s,p,N}\verti{f}_{W^{s,p}} \leq \left(\int_{0}^{\infty}t^{-sp-1} \vertii{\Delta^{\lfloor s \rfloor +1}_{t \xi} f}_{L^p(\R^N)}^p \ dt \right)^{1/p}
    \leq  C^2_{s,p,N}& \verti{f}_{W^{s,p}}, \\ & \fo \xi \in \mathbb{S}^{N-1}.
    \end{split}
   \end{flalign}
   \item For every integer $s$ and $1<p<\infty$, 
 there exist $0<C^1_{s,p,N}\leq C^2_{s,p,N}<\infty$
  such that if $f \in \dot{W}^{s,p}(\R^N)$ satisfies 
  \begin{flalign*}
    \verti{f}_{W^{s,p}}= \min \{\verti{f \circ T}_{W^{s,p}};\, \ T \in \text{SL}_N \},
  \end{flalign*}  
  then
  \begin{flalign}\label{all dir est}
    C^1_{s,p,N}\verti{f}_{W^{s,p}} \leq \left(\int_{\R^N} \verti{\partial^s_{\xi} f(x)}^p \ dx \right)^{1/p}
    \leq C^2_{s,p,N}\verti{f}_{W^{s,p}}, \ \fo \xi \in \mathbb{S}^{N-1}.
   \end{flalign}
   Same when $s=1$ and $p=1$.
   \end{enumerate}
\end{theo}

We next explain why the homogeneous Sobolev spaces $\dot{W}^{s,p}(\R^N)$ are  the natural setting for affine Sobolev \enquote{energies}. This is the content of our next result.

\begin{theo}\label{good fspace}
Let $1 \leq p < \infty$.
    \begin{enumerate}[(1)]
    \item Let $s$ be non-integer. For each $f \in L^1_{\text{loc}}(\R^N)$ , we have
\begin{flalign*}
    \verti{f}_{W^{s,p}} < \infty \iff \mathscr{E}_{s,p}(f)<\infty. 
\end{flalign*}
\item Let $s$ be an integer. For each $f \in W^{s,1}_{\text{loc}}(\R^N)$, we have
\bes
 \verti{f}_{W^{s,p}} < \infty \iff \mathscr{E}_{s,p}(f)<\infty.
\ees
\end{enumerate}
\end{theo}
We also prove the following Gagliardo-Nirenberg affine inequalities.
\begin{theo}\label{aff gag}
Let $0\leq s_1<s_2<\infty$, $1 <p_1,p_2 < \infty $, and $\theta \in (0,1)$. Set $s\coloneq \theta s_2 +(1-\theta)s_1$ and $1/p \coloneq \theta/p_2 +(1-\theta)/p_1$.  There exists $K\coloneq K_{s_1,p_1,s_2,p_2,\theta,N}<\infty$ such that
     \bes
      \mathscr{E}_{s,p}(f) \leq K \mathscr{E}_{s_1,p_1}(f)^{1-\theta} \mathscr{E}_{s_2,p_2}(f)^{\theta}, \ \fo 
 f \in \dot{W}^{s_1,p_1}(\R^N) \cap \dot{W}^{s_2,p_2}(\R^N).
     \ees
    Same when $0<s_1<s_2\leq 1$ and $1 \leq p_1,p_2<\infty$, with $s_1p_1<1$ if $s_2=1$ and $p_2=1$.
\end{theo}

Finally, we present a partial generalization of the reverse affine inequality in \cite{haddad2021affine}*{Theorem 9}, Theorem \ref{gen rev affine ineq} below. The starting point is the following.
\begin{theo}\label{start reverse}
    Let $s>0$, $1 \leq p<\infty$, and $R>0$. There exists $K=K_{s,p,R,N}<\infty$ such that we have
    \begin{flalign*}
\vertii{f}_{L^p(\R^N)}^{1-1/N}\verti{f}_{W^{s,p}}^{1/N}\leq K \verti{f\circ T}_{W^{s,p}}, 
    \end{flalign*}
    for each $T\in \text{SL}_N$ and $f\in {W}^{s,p}(\R^N)$ supported in $B(0,R)$,
    possibly except when $s\geq 2$ is an integer and $p=1$.
\end{theo}
This result, combined with Theorems \ref{there is min intro} and \ref{constraint on min intro}, allows to obtain the following.
\begin{theo}\label{gen rev affine ineq}
    Let $s>0$, $1 \leq p<\infty,$ and $R>0$. There exists $K=K_{s,p,R,N}<\infty$ such that
    \bes
    \vertii{f}_{L^p(\R^N)}^{1-1/N}\verti{f}_{W^{s,p}}^{1/N}\leq K \mathscr{E}_{s,p}(f),
    \ees
    for each $f\in W^{s,p}(\R^N)$ supported in $B(0,R)$, possibly except when $s\geq 2$ is an integer and $p=1$.
\end{theo}

In the case where $s=1$, Theorem \ref{gen rev affine ineq} reads as
   \be\label{reverse HJM}
    \vertii{f}_{L^p(\R^N)}^{1-1/N}\vertii{\nabla f}_{L^p(\R^N)}^{1/N}\leq K \mathscr{E}_{1,p}(f),
    \ee
    for each $f\in W^{1,p}(\R^N)$ supported in $B(0,R)$. This inequality is a weak version (i.e., with a non-explicit constant) of \cite{haddad2021affine}*{Theorem 9}. 
Our proof of Theorem \ref{gen rev affine ineq} is new even in the case where $s=1$. It relies on the basic AM-GM inequality, while the proof of \eqref{reverse HJM} given in  \cite{haddad2021affine}*{Theorem 9} makes strong use of the powerful Blaschke-Santal\'o inequality.

In connection with \eqref{reverse HJM},  Haddad, Jim\'enez, and Montenegro asked the following question \cite[Section 7, item (6), p. 33]{haddad2021affine}, motivated by some results on mixed variational problems in Schindler and Tintarev \cite{schindler2018compactness}: can the inequality \eqref{reverse HJM} be improved to
   \be\label{reverse HJM ext}
    \vertii{f}_{L^q(\R^N)}^{1-1/N}\vertii{\nabla f }_{L^p(\R^N)}^{1/N}\leq K \mathscr{E}_{1,p}(f)
    \ee
    for some  $q > p$ ?   We show that \eqref{reverse HJM ext} fails for any $q > p$. More generally, we present the full list of the analogues of \eqref{reverse HJM} that hold true.
\begin{theo}\label{cara rev}
Let $1 \leq p,q<\infty$, $R>0$ and $0 \leq \theta \leq 1$. 
\begin{enumerate}[(1)]
    \item  In the case where $q \leq p$, the inequality 
    \be\label{ineq 1}
    \vertii{f}_{L^q(\R^N)}^{1-\theta} \vertii{\nabla f}_{L^p(\R^N)}^{\theta} \leq K \mathscr{E}_{1,p}(f), \ \text{for each} \  f \in W^{1,p} \ \text{supported in} \ B(0,R),
    \ee
    holds for some finite $K=K_{p,q,\theta,R,N}$ if and only if $\displaystyle \theta \leq 1/N$. 
    \item In the case where $q\geq p$, the inequality
    \be \label{ineq 2}
      \vertii{f}_{L^q(\R^N)}^{1-\theta} \vertii{\nabla f}_{L^p(\R^N)}^{\theta} \leq K \mathscr{E}_{1,p}(f), \ \text{for each} \  f \in W^{1,p} \ \text{supported in} \ B(0,R),
    \ee
    holds for some finite $K=K_{p,q,\theta,R,N}$ if and only if $\displaystyle  0 \leq \theta \leq \frac{1/N+1/q-1/p}{1+1/q-1/p}$.
\end{enumerate}
\end{theo}

In particular, when $q>p$ and $\theta=1/N$, \eqref{ineq 2} does not hold. This answers negatively the question in \cite{haddad2021affine}.

\smallskip
Our text is organized as follows. In Section \ref{sect on SOb},  we  recall some standard properties of function spaces and prove Theorem \ref{there is min intro}. In Section \ref{affine en}, we study several properties of the functionals $\mathscr{E}_{s,p}$ and prove Theorem \ref{good fspace}. In Section \ref{appli}, we prove that Theorems \ref{there is min intro} and \ref{constraint on min intro} imply Theorems \ref{subcrit aff emb} and \ref{gen embedding}.
In Section \ref{pf s=1}, we illustrate our approach to Theorem \ref{constraint on min intro} in the special case where $s=1$. Section \ref{sect general fractional} is devoted to the proof of Theorem \ref{constraint on min intro} in the general case. 
Section \ref{cons} is a short discussion about the constants in Theorems \ref{subcrit aff emb}, \ref{gen embedding} and \ref{constraint on min intro} when $0<s<1$.
In Section \ref{gag sec}, we prove Theorem \ref{aff gag}.
Finally, in Section \ref{rev ineq}, we present our approach to reverse affine inequalities and prove Theorems \ref{start reverse}, \ref{gen rev affine ineq}, and \ref{cara rev}. 

\section{Sobolev semi-norms: slicing and compactness}\label{sect on SOb}
In what follows, we use the following notation.

\begin{enumerate}[(a)]
\item $N$ is the space dimension. We always assume that $N\geq 2$, unless otherwise stated. 
\item  $\verti{x}$ the Euclidean norm of $x \in \R^N$.
\item $\verti{A}$ is the Lebesgue measure of a Borel set $A \subset \R^N$.
\item $\sigma_N \coloneq \mathscr{H}^{N-1}(\mathbb{S}^{N-1})$ is the surface area of the unit sphere.
\item Given $x \in \R^N$ and $1 \leq i \leq N$, we denote $\widehat{x_i}\coloneq (x_1,\dots,x_{i-1},x_{i+1},\dots,x_N) \in \R^{N-1}$.
\item  The matrix norm is the one induced by $\verti{\ \cdot\ }$ on $\text{M}_N$.
\item Given a $k$-linear form $\eta \colon (\R^N)^k  \rightarrow \R$, we let
\begin{flalign*}
    \vertii{\eta}\coloneq \sup_{\verti{x_1}\leq 1, \dots, \verti{x_k} \leq 1}\verti{\eta(x_1,\dots,x_k)}
\end{flalign*}
This is the only norm we will consider on $k$-linear forms defined on $\R^N$.
\item Given a $k$-linear form $\eta$ and a matrix $T$, we denote by $T^{\ast}\eta$ the $k$-linear form
\be\label{pull b}
 (\R^N)^k\ni (\xi^1, \dots, \xi^k) \mapsto   T^{\ast}\eta (\xi^1, \dots, \xi^k)\coloneq  \eta(T(\xi_1), \dots, T(\xi_k)).
\ee
\item Given $f\colon \R^N \rightarrow \R$ a measurable function and $h \in \R^N$, 
we let  
\begin{flalign*}
  \R^N\ni   x \mapsto \Delta_h f(x)\coloneq f(x+h)-f(x).  
\end{flalign*}
Given $m\geq 1$ an integer, we define higher-order difference operators by $\Delta^{m+1}_h= \Delta_h \circ \Delta^{m}_h$, so that 
\begin{flalign}\label{useful identity}
    (\Delta^m_hf)(x) = \sum_{l=1}^m \binom{m}{l}(-1)^{m-l}f(x+lh),\ \fo x\in\R^N.
\end{flalign}      

\item Given $s$ an integer, a function $f$ in the Sobolev space $W^{s,1}_{\text{loc}}(\R^N)$, and $\xi \in \R^N$, 
we denote by $\partial_{\xi}^{s}f$ the function (defined for a.e.\ $x \in \R^N$)
\bes
\partial^s_{\xi} f(x) \coloneq D_x^s f(\xi, \dots, \xi) = \sum_{\verti{\alpha}=s} \xi^{\alpha} \partial^{\alpha}f(x).
\ees
\item  Let $s$ be non-integer. We denote by $\dot{W}^{s,p}=\dot{W}^{s,p}(\R^N)$ the space of functions 
$f \in L^1_{\text{loc}}(\R^N)$ such that 
$\verti{f}_{W^{s,p}}<\infty$, where
\be\label{def 1}
\ba 
\verti{f}_{W^{s,p}}^p &\coloneq \int_{\R^N} \frac{\vertii{\Delta_h^{\lfloor s \rfloor +1}f}_{L^p}^p}{\verti{h}^{sp+N}} \ dh \\ &= \int_{\mathbb{S}^{N-1}} \left(\int_{0}^{\infty}t^{-sp-1}\vertii{\Delta^{\lfloor s \rfloor +1}_{t\xi}f}_{L^p}^p  \ dt \right) \ \  d\mathscr{H}^{N-1}(\xi).
\ea
\ee

\item  Let $s$ be an integer. We denote by $\dot{W}^{s,p}=\dot{W}^{s,p}(\R^N)$ the space of functions $f \in W_{\text{loc}}^{s,1}(\R^N)$ such that $\verti{f}_{W^{s,p}}<\infty$, where
\be \label{def 2}
\verti{f}_{W^{s,p}}^p \coloneq \int_{\R^N}\vertii{D_x^s f}^p \ dx.
\ee
In particular,
\bes
\verti{f}_{W^{s,p}}=\vertii{\nabla f}_{L^p}.
\ees
\item We set, for convenience, $\mathscr{E}_{0,p}(f)\coloneq \vertii{f}_{L^p}$ and $\verti{f}_{W^{0,p}}\coloneq \vertii{f}_{L^p}$, for each measurable $f$.
\item The semi-norms $\verti{\cdot}_{W^{s,p}}$ are invariant under orthogonal transformations:  for each $s>0$, $1 \leq p<\infty$, $f \in W^{s,p}$, and $R \in \text{O}_N$, we have
 \begin{flalign*}
     \verti{f\circ R}_{W^{s,p}}=\verti{f}_{W^{s,p}}.
 \end{flalign*}
\item If $s>0$ and $1 \leq p<\infty$ are such that $sp<N$, we set $\displaystyle q\coloneq \frac{Np}{N-sp}$ and denote 
\begin{flalign*}
    \mathring{W}^{s,p}\coloneq \{f \in L^q;\, \verti{f}_{W^{s,p}}<\infty \}.
\end{flalign*}
\item In what follows, $\rho$ stands for a standard mollifier and we set $\displaystyle \rho_{\delta}(x)\coloneq1/\delta^N \rho(\cdot / \delta^N),$ for each $\delta>0$.
\end{enumerate}

We next recall or establish some basic estimates for Sobolev semi-norms.   The first one is obvious. See, e.g., Leoni \cite[Theorem 6.62]{leoni2023first} and \cite[Lemma 17.25]{leoni2024first} for the second and third ones.
\begin{lemma}\label{decre convol}
    For each $1 \leq p<\infty$, integer $m$ , $h \in \R^N$, and $f \in L^1_{\text{loc}}=L^1_{\text{loc}}(\R^N)$, we have 
\begin{flalign*}
&\Delta^m_h(f\ast \rho) = (\Delta^m_h f)\ast \rho,
 \\ &\vertii{\Delta^m_{h}(f\ast \rho)}_{L^p} \leq \vertii{\Delta^m_{h} f}_{L^p}.
\end{flalign*}
\end{lemma}

\begin{lemma}\label{ast conv}
    Let $0<s<1$ and $1 \leq p<\infty$. For each $f \in \dot{W}^{s,p}$, we have 
    \begin{flalign*}
        \verti{f\ast \rho_{\delta}-f}_{W^{s,p}}{\rightarrow} \ 0 \  \text{as}\ \delta\to 0. 
    \end{flalign*}
\end{lemma}

\begin{lemma}\label{delta fo}
    Let $\chi_1\coloneq \mathbbm{1}_{[0,1]}$ and for $m\geq 2$, set \bes \chi_m\coloneq\chi_1 \ast  \dots \ast \chi_1 \ \text{($m$ times)}.\ees
    For each $\varphi \in C^{\infty}(\R^N)$, integer $m$, and $h \in \R^N$, we have
    \begin{flalign}\label{explicit diff}
        (\Delta_{h}^m \varphi)(x)=\int_{0}^m \chi_m(t) D^m_ {x+t h}\varphi(h,\dots,h) \ dt
    \ , \ \  \fo x \in \R^N.
    \end{flalign}
\end{lemma}
We next recall a few slicing inequalities involving  semi-norms. For $s=1$, we have the obvious inequalities, for each measurable $f$ and each orthonormal basis $(u_1,\dots,u_N)$ of $\R^N$,
\begin{flalign}\label{slice s=1}
    \frac{1}{N} \sum_{i=1}^N \vertii{\partial_{u_i} f}_{L^p} \leq \vertii{\nabla f}_{L^p} \leq  \sum_{i=1}^N \vertii{\partial_{u_i} f}_{L^p},
\end{flalign}
the quantities above being infinite if $f \notin \dot{W}^{1,p}$.
For other values of $s$, we mention the following counterparts of \eqref{slice s=1}, for which we refer the reader to, e.g., Triebel \cite{triebel2010theory}.
\begin{theo}\label{theo besov} (\cite{triebel2010theory}*{Theorem, Section 2.5.13})
Let $s$ be non-integer and $1 \leq p < \infty$.
There exist $0<K^1_{s,p,N}\leq K^2_{s,p,N} <\infty$ such that, for each $f \in \dot{W}^{s,p}$ and each orthonormal basis $(u_1,\dots,u_N)$ of $\R^N$, we have
\be\label{Besov}
\ba
   K^1_{s,p,N}&\sum_{i=1}^N \left(\int_{0}^{\infty}t^{-sp-1} \vertii{\Delta^{\lfloor s \rfloor +1}_{t u_i} f}_{L^p}^p \ dt \right)^{1/p} \leq \verti{f}_{\dot{W}^{s,p}} \\ & \hspace{70 pt} \leq
    K^{2}_{s,p,N}\sum_{i=1}^N \left(\int_{0}^{\infty}t^{-sp-1} \vertii{\Delta^{\lfloor s \rfloor +1}_{t u_i}f}_{L^p}^p \ dt \right)^{1/p}.
\ea    
\ee
In particular, for each $\xi \in \mathbb{S}^{N-1}$, we have
\be
K^1_{s,p,N} \pari{\int_{0}^{\infty}t^{-sp-1} \vertii{\Delta_{t\xi}^{\lfloor s \rfloor +1} f}_{L^p}^p \ dt}^{1/p} \leq \verti{f}_{W^{s,p}}.
\ee
Moreover, 
when $0<s<1$, the above inequalities hold for each measurable $f$.
\end{theo}
\begin{rema}\label{co s 1}
When $0<s<1$ and $1 \leq p<\infty$, one may choose, in \eqref{Besov}, constants independent of $s$ and $p$: $K^1_{s,p,N}=K^1_N >0$ and $K^2_{s,p,N}=K^2_N<\infty$. Although this fact is not explicitly stated in Leoni \cite{leoni2023first},  it follows from the proof of \cite[Theorem 6.35]{leoni2023first}.
\end{rema}
\begin{theo}\label{integer slicing}
    Let $s$ be an integer and  $1<p<\infty$.  There exist $0<K^1_{s,p,N}\leq K^2_{s,p,N}<\infty$ such that, for each $f \in \dot{W}^{s,p}$ and each orthonormal basis $(u_1,\dots,u_N)$ of $\R^N$, we have
    \be\label{ttt1} \begin{split}
K^1_{s,p,N}\sum_{i=1}^{N} \vertii{\partial^s_{u_i} f}_{L^p} \leq \verti{f}_{W^{s,p}}
    \leq K^2_{s,p,N} \sum_{i=1}^{N}\vertii{\partial^s_{u_i}f}_{L^p}. 
    \end{split}
    \ee
\end{theo}

 The first inequality in \eqref{ttt1} is obvious, and was stated only in order to highlight the analogy between the two theorems. The non-trivial assertion in \eqref{ttt1} is the second inequality. For its validity, when $s\geq 2$, 
 the assumption $p\neq 1$ is necessary. Indeed, Ornstein's  family of counterexamples \cite{ornstein1962non} shows that, when $p=1$ and $s\ge 2$ is an integer, the second inequality in \eqref{ttt1} fails. Theorem \ref{integer slicing} may be obtained as a consequence of its inhomogeneous counterpart \cite[Theorem, Section 2.5.6]{triebel2010theory}, see Appendix \ref{A}. 
\begin{theo}\label{higher slicing}(\cite[Theorem 6.61]{leoni2024first})
    Let $s$ be non-integer and $1\leq p<\infty$. There exists $A_{s,p,N}<\infty$ such that, for each measurable $f$ and each orthonormal basis $(u_1,\dots,u_N)$ of $\R^N$, we have
    \bes
     \ba
     \left(\int_{\R^N}\frac{\vertii{\Delta^{N(\lfloor s \rfloor +1)}_{h}f}_{L^p}^p}{\verti{h}^{sp+N}} \ dh \right)^{1/p} \leq A_{s,p,N} \sum_{i=1}^N\pari{\int_{0}^{\infty}t^{-sp-1}\vertii{\Delta^{\lfloor s \rfloor +1}_{tu_i}f}_{L^p}^p \ dt }^{1/p}.
     \ea
    \ees
    \end{theo}
\begin{rema}\label{1d norm}
    Let $s>0$ and $1  \leq p<\infty$. 
\begin{enumerate}[(1)]
    \item If $s$ is non-integer, we have
     \bes
      2 \int_{0}^{\infty}t^{-sp-1} \vertii{\Delta^{\lfloor s \rfloor +1}_{te_i} f}_{L^p}^p \ dt = \int_{\R^{N-1}} \verti{f(x_1,\dots,x_{i-1},\cdot,x_{i+1},\dots,x_N)}_{W^{s,p}(\R)}^p  d \widehat{x_i},   
     \ees
     for each $1 \leq i \leq N$ and $f \in L^1_{\text{loc}}$.
    \item If $s$ is an integer, we have 
    \bes
     \vertii{\partial^s_i f}_{L^p}^p = \int_{\R^{N-1}} \verti{f(x_1,\dots,x_{i-1},\cdot,x_{i+1},\dots,x_N)}_{W^{s,p}(\R)}^p  d \widehat{x_i}, 
    \ees
    for each $1\leq i \leq N$ and $f \in W^{s,1}_{\text{loc}}$.
\end{enumerate}
\end{rema}

\smallskip
We now state useful change of variable formulas.
\begin{lemma}\label{chan}
For each integer $m$, measurable $f$ and linear map $U$, we have 
\be\label{delta comp}
\Delta^m_h\left(f \circ U \right)=\left(\Delta^m_{U(h)} f \right)\circ U.
\ee
\end{lemma}
In particular, we have the following.
\begin{lemma}\label{frac chan}
Let $m$ be an integer, $s$ be non-integer and $1 \leq p<\infty$. Let $(u_1,\dots,u_N)$ be an orthonormal basis of $\R^N$, $\text{O} \in \text{O}_N$ be defined by $\text{O}u_i=e_i$, and $\text{D}=\text{diag}(\lambda_1,\dots,\lambda_N)$ be invertible. 
We have
\bes
\int_{0}^{\infty}t^{-sp-1}\vertii{\Delta^{m}_{tu_i} (f\circ (\text{DO}))}^p_{L^p} \ dt= \frac{\verti{\lambda_i}^{sp}}{\verti{\text{det} \ \text{D}}}  \int_{0}^{\infty}t^{-sp-1}\vertii{\Delta^{m}_{te_i} f}^p_{L^p} \ dt, 
\ees
for each $1 \leq i \leq N$ and $f$ measurable.
\end{lemma}
We next state some auxiliary results  that we will use in the proof of Theorem \ref{there is min intro}.
\begin{lemma}\label{no constant direction}
    Let $m$ be an integer and $1 \leq p,q < \infty$. For each $f$  in $L^q\setminus\{0\}$, 
    there exist $\delta=\delta_f>0$ and $C=C_f>0$ such that
    \begin{flalign*}
        \vertii{\Delta^m_{t\xi} f}_{L^p} \geq C t^m, \ \fo 0<t<\delta, \  \fo \xi \in \mathbb{S}^{N-1}.
    \end{flalign*}
    Same when $0<s<\infty$ and $f \in \dot{W}^{s,q}$  satisfies $\verti{f}_{W^{s,q}}> 0$.
\end{lemma}

\begin{lemma}\label{no cons direc toy}
    Let $1 \leq p < \infty$.
    \begin{enumerate}[(1)]
    \item Let $s$ be non-integer.  If $f \in \dot{W}^{s,p}$  is such that
    \begin{flalign*}
        \inf_{\xi \in \mathbb{S}^{N-1}} \int_{0}^{\infty}t^{-sp-1} \vertii{\Delta_{t\xi}^{\lfloor s \rfloor +1} f}_{L^p}^p dt =0,
    \end{flalign*}
    then $\verti{f}_{W^{s,p}}=0$.
    \item Let $ s $ be an integer.
    If $f \in \dot{W}^{s,p}$ is such that 
    \begin{flalign*} 
    \inf_{\xi \in \mathbb{S}^{N-1}} \int_{\R^N} \verti{\partial_{\xi}^s f(x)}^p \ dx = 0,
    \end{flalign*}
    then $\verti{f}_{W^{s,p}}=0$.
\end{enumerate}
\end{lemma}

\begin{lemma}\label{Conv compo matrix}
    Let $s>0$ and $1 \leq p<\infty$. Let $f \in \dot{W}^{s,p}(\R^N)$ and let $(T_n) \subset \text{GL}_N$ converge to a matrix $T \in \text{GL}_N$. Then 
    \begin{flalign*}
         \verti{f\circ T_n}_{W^{s,p}}\to \verti{f\circ T}_{W^{s,p}} \ \text{as} \ n \rightarrow \infty.
    \end{flalign*}
\end{lemma} 
We  now turn to the proofs of Lemmas \ref{no constant direction}, \ref{no cons direc toy}, \ref{Conv compo matrix}, and Theorem \ref{there is min intro}.
The essential ingredient in the proof of Lemma \ref{no constant direction} is the following trivial fact.
\begin{lemma}\label{poly L=0}
 Let $1\leq q<\infty$. If $g \in L^q$ is such that 
\begin{flalign*}
    \R \ni x_1 \mapsto g(x_1,x_2,\dots,x_N)
\end{flalign*}
is a polynomial for a.e.\ $(x_2,\dots,x_N) \in \R^{N-1}$, then $g=0$  a.e. 
\end{lemma}

\begin{proof}[Proof of Lemma \ref{no constant direction}]
    We have to show  that if $f \in L^q$ , respectively $f \in \dot{W}^{s,q}$, is such that there exist sequences
$t_n \searrow 0$ and $(\xi_n) \subset \mathbb{S}^{N-1}$ satisfying
\begin{flalign}\label{contra direc}
    \vertii{\Delta^m_{t_n\xi_n}f}_{L^p} < \frac{t^m_n}{n+1}, \ \fo n,
\end{flalign}
then $\vertii{f}_{L^q}=0$, respectively $\verti{f}_{W^{s,q}}=0$.

\smallskip
In both cases, we may assume, without loss of generality, that $\xi_n \rightarrow e_1$.   
We claim that we may further assume that $f \in C^{\infty}(\R^N)$. Indeed, by Lemma \ref{decre convol}, $f\ast\rho_\delta$ satisfies
\begin{flalign*}
    \vertii{\Delta^m_{t_n\xi_n}(f\ast \rho_{\delta})}_{L^p} < \frac{t^m_n}{n+1}, \fo n.
\end{flalign*} 

On the other hand, as $\delta\to 0$, we have
\begin{flalign*}
    &\vertii{f \ast \rho_{\delta}}_{L^q} \to \vertii{f}_{L^q}, \ \text{in the case where} \ f \in L^q ,
    \\ & \verti{f \ast \rho_{\delta}}_{W^{s,q}} \to \verti{f}_{W^{s,q}}, \ \text{in the case where} \ f \in \dot{W}^{s,q}.
\end{flalign*}
(For the second assertion, we rely on Lemma \ref{ast conv}.)

\smallskip
Therefore, by replacing $f$ with $f\ast\rho_\delta$, then passing to the limits, it suffices to deal with the case where $f$ is smooth.

\smallskip

Consider now a smooth function satisfying \eqref{contra direc}. 
By Lemma \ref{delta fo}, we have 
\be \label{pt conv}
\d \frac{\Delta_{t_n \xi_n}^m f (x)}{t_n^m}\to \partial_1^m f(x), \ \text{pointwise as} \ n \rightarrow \infty.
\ee Fatou's lemma, combined with \eqref{contra direc} and \eqref{pt conv}, implies that  $\partial_1^m f=0$. If $f\in L^q$, then Lemma \ref{poly L=0} implies that $f=0$, and we are done.

\smallskip
When $f \in \dot{W}^{s,q}$, we argue as follows. If $s$ is an integer then, for each $\alpha$ such that $\verti{\alpha}=s$ and each $(x_2,\dots,x_N) \in \R^{N-1}$,  the function
    \begin{flalign*}
        x_1 \mapsto \partial^{\alpha}f(x_1,x_2,\dots,x_N)
    \end{flalign*}
    is a polynomial of degree $\leq m-1$. This follows from the Schwarz lemma, which yields 
    \begin{flalign*}
        \partial_1^m(\partial^{\alpha}f)=\partial^{\alpha}(\partial_1^m f)=0.
    \end{flalign*}
    
Therefore, Lemma \ref{poly L=0} implies that $\vertii{\partial^{\alpha} f}_{L^q}=0$ and thus $\verti{f}_{W^{s,q}}=0$.

\smallskip
If $s$ is non-integer, then, by Theorem \ref{theo besov}, for each $1\leq i \leq N $, we have 
\begin{flalign*}
    \int_{0}^{\infty}t^{-sq-1}\vertii{\Delta^{\lfloor s \rfloor +1}_{t
     e_i} f}_{L^q}^q \ dt < \infty,
\end{flalign*}
which implies that
\begin{flalign*}
     \Delta^{\lfloor s \rfloor +1}_{t
    e_i} f \in L^{q}, \ \text{for a.e.} \ t>0.
\end{flalign*}
On the other hand, for each $1 \leq i \leq N$, $(x_2,\dots,x_N) \in \R^{N-1}$, and $t>0$, the function
\begin{flalign*}
    x_1 \mapsto \Delta^{\lfloor s \rfloor +1}_{t e_i}f(x_1,x_2,\dots,x_N)
\end{flalign*} 
is a polynomial, since
\begin{flalign*}
    \partial^m_1(\Delta^{\lfloor s \rfloor +1}_{t e_i}f)
     = \Delta^{\lfloor s \rfloor +1}_{t e_i}(\partial^m_1 f )=0.
\end{flalign*}
Hence, an application of Lemma \ref{poly L=0} gives that, for each $1 \leq i \leq N$,
\begin{flalign*}
    \vertii{\Delta^{\lfloor s \rfloor +1}_{t e_i} f}_{L^q}=0, \ \text{for a.e.} \ t >0,
\end{flalign*}
and therefore
\begin{flalign*}
    \int_{0}^{\infty}t^{-sq-1}\vertii{\Delta^{\lfloor s \rfloor +1}_{t
     e_i} f}_{L^q}^q \ dt=0.
\end{flalign*}
Theorem \ref{theo besov} hence implies that $\verti{f}_{W^{s,q}}=0$.

\smallskip
This completes the proof of Lemma \ref{no constant direction}.
\end{proof}
\begin{proof}[Proof of Lemma \ref{no cons direc toy}]
    (1) Let $s$ be non-integer and let $f \in \dot{W}^{s,p}$ be such that $\verti{f}_{W^{s,p}}\neq 0$. By Lemma \ref{no constant direction}, there exist $\delta>0$ and $C>0$ such that
    \begin{flalign*}
        \vertii{\Delta^{\lfloor s \rfloor +1}_{t\xi} f}_{L^p} \geq C t^{\lfloor s \rfloor +1}, \ \fo 0<t<\delta, \  \fo \xi \in \mathbb{S}^{N-1}.
    \end{flalign*}
    Hence,  for each $\xi \in \mathbb{S}^{N-1}$, we have 
    \bes
    \ba
     \int_{0}^{\infty} t^{-sp-1}\vertii{\Delta^{\lfloor s \rfloor +1}_{t\xi} f}_{L^p}^p \ dt &\geq \int_{0}^{\delta} t^{-sp-1}\vertii{\Delta^{\lfloor s \rfloor +1}_{t\xi} f}_{L^p}^p \ dt \\ &\geq C^p \int_{0}^{\delta}t^{(\lfloor s \rfloor +1)p -sp-1} \ dt = \frac{C^p\delta^{(\lfloor s \rfloor +1 -s)p}}{(\lfloor s \rfloor +1 -s)p}.
    \ea
    \ees
Therefore, 
\bes
\inf_{\xi \in \mathbb{S}^{N-1}} \int_{0}^{\infty} t^{-sp-1}\vertii{\Delta^{\lfloor s \rfloor +1}_{t\xi} f}_{L^p}^p \ dt  \geq \frac{C^p\delta^{(\lfloor s \rfloor +1 -s)p}}{(\lfloor s \rfloor +1 -s)p}>0.
\ees
This completes the proof of $(1)$. 

\smallskip
\noindent
(2) Let $s$ be an integer and let $f \in \dot{W}^{s,p}$ be such that 
\bes
\inf_{\xi \in \mathbb{S}^{N-1}} \int_{\R^N} \verti{\partial_{\xi}^{s} f(x)}^p \ dx =0. 
\ees

Without loss of generality, we may assume that there exists a sequence $(\xi_n) \subset \mathbb{S}^{N-1}$ such that $\xi_n \to e_1$ and
\bes
\int_{\R^N} \verti{\partial_{\xi_n}^s f (x) }^p \ dx \to 0.
\ees
Since 
\bes
\partial^s_{\xi_n} f(x) \to \partial^s_1 f (x), \ \text{for a.e.} \ x \in \R^N,
\ees
Fatou's lemma yields
\bes
\int_{\R^N}\verti{\partial^s_1 f(x)}^p \ dx=0, \ \text{and therefore} \ \ \partial^s_{1}f(x)=0, \ \text{for a.e.} \ x \in \R^N.
\ees

Using this fact, we find, as in the proof of Lemma \ref{no constant direction}, that $\verti{f}_{W^{s,p}}=0$.
\end{proof}
\begin{proof}[Proof of Lemma \ref{Conv compo matrix}]
    Let $f \in \dot{W}^{s,p}$ and let $(T_n) \subset \text{GL}_N$ converge to
     a matrix $T \in \text{GL}_N$.

\smallskip
    If $s$ is non-integer, we argue as follows. Using
    \eqref{delta comp},
    we find that 
    \bes
    \ba
        \verti{f \circ T_n}_{W^{s,p}}^p&= \int_{\R^N} \frac{\vertii{\Delta^{\lfloor s \rfloor +1}_h (f\circ T_n)}_{L^p}^p}{\verti{h}^{sp+N}} \ dh=\int_{\R^N} \frac{\vertii{\pari{\Delta^{\lfloor s \rfloor +1}_{T_n(h)} f}\circ T_n}_{L^p}^p}{\verti{h}^{sp+N}} \ dh
   \\&=\int_{\R^N} \frac{\vertii{\pari{\Delta^{\lfloor s \rfloor +1}_{T_n(h)} f}}_{L^p}^p}{\verti{h}^{sp+N}} \ \frac{dh}{\verti{\text{det} \ T_n}} =  \int_{\R^N} \frac{\vertii{\Delta^{\lfloor s \rfloor +1}_{z} f}_{L^p}^p}{\verti{T_n^{-1}(z)}^{sp+N}} \ \frac{dz}{\verti{\text{det}\ T_n}^2}.
    \ea
    \ees

\smallskip
Since $T_n \rightarrow T$, we have
\bes 
\text{det} \ T_n \to \text{det} \ T , \ \ \ T_n^{-1}(z) \to T^{-1}(z), \ \fo z \in \ \R^N,
\ees
and there exists $C>0$ such that
\bes
\verti{T_n^{-1}(z)}^{sp+N} \verti{\text{det} \ T_n}^2 \geq C\verti{z}^{sp+N}, \ \fo z \in \R^N, \ \fo n \in \N.
\ees

The conclusion of the lemma then follows by dominated convergence.

\smallskip

If $s$ is an integer, we rely on the following facts. Fact 1: given $g \in \dot{W}^{s,p}$ and $L \in \text{GL}_N$, we have 
    \begin{flalign}\label{chainr}\begin{split}
        D_x^s(g \circ L)(u^1, \dots, u^s )&= D_{L(x)}g(L(u^1),\dots, L(u^s))= L^{*}(D_{L(x)}g)(u^1,\dots,u^s)
    \end{split}\end{flalign}
    for each $(u^1, \dots,u^s) \in (\R^N)^s$ and a.e.\ $ x \in \R^N.$ (Recall that the notation $T^{*}\eta$ was introduced in \eqref{pull b}.)  Fact 2: given  a $k$-linear form $\eta$ on $\R^N$ and a sequence $(L_n) \subset M_N$ that converges to $L \in M_N$ ,
    we have
    \begin{flalign*} 
        & \vertii{L_n^{\ast} \eta }\leq  C \vertiii{L_n}^k \vertii{\eta}, \\
        & \vertii{L_n^{\ast}\eta-L^{\ast} \eta } \rightarrow 0. 
    \end{flalign*}
Using these elementary facts, we find that    
\begin{flalign*}\label{alt form}\begin{split}
        \verti{f\circ T_n}_{W^{s,p}}^p &= \int_{\R^N} \vertii{D^s_x(f\circ T_n)}^p \ dx 
        \\ &= \int_{\R^N} \vertii{T_n^{\ast}(D^s_{T_n(x)} f) }^p\ dx=  \int_{\R^N} \vertii{T_n^{\ast}(D^s_{y} f) }^p\ \frac{dy}{\verti{\text{det} \ T_n}}.
\end{split}
\end{flalign*}
     (Here we rely on Fact 1 for the second equality, and we make the change of variables $y=T_n(x)$ in order to obtain the last one.)  Moreover, 
 since $T_n \to T$, we have, by Fact 2,
        \bes
             \frac{\vertii{T_n^{\ast}(D^s_{y}f)}^p} {\verti{\text{det} \ T_n}} \leq \frac{\vertiii{T_n}^{sp}}{\verti{\text{det} \ T_n}} \vertii{D^s_{y}f}^p \leq C\vertii{D^s_{y}f}^p,\ \fo y\in \R^N,\ \fo n, 
        \ees
        where $C>0$ is independent of $n$. On the other hand, Fact 2 yields
 \bes
 \frac{\vertii{T_n^{\ast}(D^s_{y} f)}^p}{\verti{\text{det} \ T_n}}\to \frac{\vertii{T^{\ast}(D^s_{y} f)}^p}{\verti{\text{det}\ T}}, \ \fo y \in \R^N.
 \ees

 By dominated convergence, we find that 
    \bes
        \verti{f\circ T_n}_{W^{s,p}}^p \to \int_{\R^N} \vertii{T^{\ast}(D^s_{y} f)}^p\ \frac{dy}{\verti{\text{det} \ T}}= \verti{f \circ T}_{W^{s,p}}^p. \qedhere
    \ees
\end{proof}
We now prove Theorem \ref{there is min intro}.
\begin{proof}[Proof of Theorem \ref{there is min intro}] Consider $f \in \dot{W}^{s,p}$ and a minimizing sequence $(T_n) \subset \text{SL}_N$ such that
    \begin{flalign*}
        \verti{f \circ T_n}_{W^{s,p}}\rightarrow \inf\{\verti{f\circ T}_{W^{s,p}};\,  T \in \text{SL}_N \}. 
    \end{flalign*}
    Without loss of generality, we may assume that $\verti{f}_{W^{s,p}} \neq 0$. 

    \smallskip
    \noindent \textit{Claim}. $(T_n)$ is bounded. 
    
    \smallskip
    Granted the claim, we complete the proof of  Theorem \ref{there is min intro} as follows. 
Consider a subsequence $(T_{n_k})$ and $T\in \text{SL}_N$ such that $T_{n_k} \to T$.
By Lemma \ref{Conv compo matrix}, the conclusion of the theorem holds with $T_f=T$.

\smallskip
We now prove the claim, which amounts to the existence of $M<\infty$ such that
\bes
    \verti{T_n(\xi)} \leq M, \ \fo \xi \in \mathbb{S}^{N-1}, \fo n \in \N.
\ees

 If $s$ is non-integer, we denote $m:=\lfloor s \rfloor +1$. On the one hand, an application of Theorem \ref{theo besov} yields the existence of $C<\infty$ such that
    \begin{flalign}\label{upper b}
        \int_{0}^{\infty}t^{-sp-1}\vertii{\Delta^{m}_{t\xi}(f\circ T_n)}_{L^p}^p \ dt \leq C \verti{f\circ T_n}_{W^{s,p}}^p, \fo \xi \in \mathbb{S}^{N-1}, \fo n \in \N.
    \end{flalign}
On the other hand, using \eqref{delta comp}, we find that, with $w_n\coloneq \d\frac{1}{\verti{T_n(\xi)}}T_n(\xi)$, we have
\be
\label{calcul change of var}
\ba
    \int_{0}^{\infty}t^{-sp-1}\vertii{\Delta^{m}_{t\xi}(f\circ T_n)}_{L^p}^p \ dt&=  \int_{0}^{\infty}t^{-sp-1}\vertii{\pari{\Delta^{m}_{T_n(t\xi)}f}\circ T_n }_{L^p}^p \ dt
    \\ &=\int_{0}^{\infty}t^{-sp-1}\vertii{\Delta^{m}_{T_n(t\xi)}f }_{L^p}^p \ dt
    \\ 
  &= \verti{T_n(\xi)}^{sp}\int_{0}^{\infty}u^{-sp-1}\vertii{\Delta^m_{uw_n}f }_{L^p}^p\ du.
\ea
\ee

Combining \eqref{upper b} and \eqref{calcul change of var}, we find that, for each $n \in \N$ and $\xi \in \mathbb{S}^{N-1}$,
\bes
\verti{T_n(\xi)}^{sp}\int_{0}^{\infty}u^{-sp-1}\vertii{\Delta^m_{uw_n}f}_{L^p}^p\ du  \leq C \verti{f\circ T_n}_{W^{s,p}}^p.
\ees
Since $\verti{f}_{W^{s,p}} \neq 0$, Lemma \ref{no cons direc toy} yields
\begin{flalign*}
\alpha\coloneq \inf_{\omega \in \mathbb{S}^{N-1}}\int_{0}^{\infty}u^{-sp-1}\vertii{\Delta^m_{t\omega}f}_{L^p}^p \ du>0.
\end{flalign*}
Therefore, we have
\begin{flalign*}
    \verti{T_n(\xi)}^{sp} \leq \frac{C \sup_n \verti{f\circ T_n}_{W^{s,p}}^p}{\alpha} < \infty , \ \fo \xi \in \mathbb{S}^{N-1},\,  \fo n \in \N,
\end{flalign*}
since $(\verti{f\circ T_n}_{W^{s,p}})$ is bounded.
This proves the claim in the case where $s$ is non-integer.

\smallskip
We next consider  the case where $s$ is an integer. The inequality 
\bes
\verti{\partial^s_{\xi}f(x)}\leq \vertii{D^s_x f}, \ \text{for a.e.} \ x \in \R^N, \ \text{for each} \ \xi \in \mathbb{S}^{N-1},
\ees
yields  
\begin{flalign*}
    \int_{\R^N} \verti{\partial^s_{\xi}(f\circ T_n)(x)}^p \ dx \leq \verti{f\circ T_n}_{W^{s,p}}^p, \ \fo \xi \in \mathbb{S}^{N-1}.
\end{flalign*}
But we also have 
\begin{flalign*}\begin{split}
    \int_{\R^N} \verti{\partial^s_{\xi}(f\circ T_n)(x)}^p \ dx &= \int_{\R^N} \verti{D^s_{T_n(x)}f(T_n(\xi), \dots, T_n(\xi))}^p \ dx
\\ & = \verti{T_n(\xi)}^{sp}\int_{\R^N}\verti{D^s_{T_n(x)}f\left(w_n, \dots,w_n \right)}^p \ dx
\\ & = \verti{T_n(\xi)}^{sp}\int_{\R^N}\verti{D^s_y f\left(w_n, \dots,w_n \right)}^p \ dy.
\end{split}
\end{flalign*}

As in the fractional case, combining these two facts, using the boundedness of the sequence $(\verti{f\circ T_n}_{W^{s,p}}),$ and applying Proposition \ref{no cons direc toy} yields
\bes
    \verti{T_n(\xi)}^{sp} \leq \frac{C\sup_n \verti{f \circ T_n}^p_{W^{s,p}}}{\d\inf_{\omega \in \mathbb{S}^{N-1}} \int_{\R^N} \verti{\partial^s_{\omega}f(x)}^p\ dx} < \infty.
\ees

This proves the claim in the case where $s$ is an integer and completes the proof of Theorem \ref{there is min intro}.
\end{proof}

\section{A closer look at affine 
\enquote{energies}} \label{affine en}
Given a bijective convex function $\Psi \colon [0,\infty] \to [0,\infty]$, we may try to define \enquote{refinements} of Sobolev semi-norms as follows. We consider
\bes
\left[\mathscr{E}_{s,p}^{\Psi}(f)\right]^p\coloneq  \sigma_N\Psi\left(\frac{1}{\sigma_N}\int_{\mathbb{S}^{N-1}} \Psi^{-1}\left(\int_{0}^{\infty}t^{-sp-1}  \vertii{\Delta^{\lfloor s \rfloor + 1}_{t\xi} f}_{L^p}^p \ dt\right) \ d\mathscr{H}^{N-1}(\xi)\right),
\ees
for each measurable $f$ and $s$ non-integer, and
\bes
\left[\mathscr{E}_{s,p}^{\Psi}(f)\right]^p\coloneq \sigma_N\Psi \left(\frac{1}{\sigma_N}\int_{\mathbb{S}^{N-1}} \Psi^{-1}\left(\int_{\R^N} \verti{\partial^s_{\xi} f(x)}^p \ dx \right) \ d\mathscr{H}^{N-1}(\xi) \right),
\ees
for each $f \in W^{s,1}_{\text{loc}}$ and $s$ integer. Given $s$ an integer, we also set
\be\label{radial norm}
\verti{f}^{*}_{W^{s,p}}\coloneq \left(\int_{\mathbb{S}^{N-1}}  \left(\int_{\R^N} \verti{\partial^s_{\xi} f(x)}^p \ dx \right) \ d\mathscr{H}^{N-1}(\xi)\right)^{1/p}, 
\ee
for each $f \in W^{s,1}_{\text{loc}}$. $\verti{\ \cdot \ }^{*}_{W^{s,p}}$ is a semi-norm which is equivalent to $\verti{\ \cdot \ }_{W^{s,p}}$ (see Lemma \ref{int norm}). 

\smallskip
By Jensen's inequality, we have, for each measurable $f$,
\be\label{frac J}
\mathscr{E}_{s,p}^{\Psi}(f)\leq \verti{f}_{W^{s,p}},
\ee
in the case where $s$ is non-integer, respectively, for each $f \in W^{s,1}_{\text{loc}}$,
\be\label{int J}
\mathscr{E}_{s,p}^{\Psi}(f) \leq \verti{f}^{*}_{W^{s,p}},
\ee
in the case where $s$ is an integer. On the other hand, 
\be\label{rad}
\eqref{frac J} \ \text{and} \ \eqref{int J} \ \text{are equalities for radial functions.} 
\ee
\smallskip
In the special case where $\Psi=\Psi_{s,p}$, with
\bes
\Psi_{s,p} \colon [0,\infty] \ni x \mapsto 
\begin{cases}
x^{-sp/N},& \ \text{if} \ x \in (0,\infty) \\
\infty, & \ \text{if} \ x=0 \\
0,& \ \text{if} \ x=\infty,
\end{cases}
\ees
we obtain
\be\label{sp}
\mathscr{E}_{s,p}=  \mathscr{E}_{s,p}^{\Psi}.
\ee

In particular, we have
\begin{lemma}\label{Jensen}
Let $1 \leq p <\infty $. 
\begin{enumerate}[(1)]
    \item Let $s$ be non-integer. We have  
\bes
\mathscr{E}_{s,p}(f) \leq  \verti{f}_{W^{s,p}},
\ees
for each $f \in L^1_{\text{loc}}.$
\item Let $s$ be an integer. We have
\bes
\mathscr{E}_{s,p}(f) \leq \verti{f}^{*}_{W^{s,p}},
\ees
for each $f \in W^{s,1}_{\text{loc}}$ ($\verti{\ \cdot \ }_{W^{s,p}}^{*}$ is defined in \eqref{radial norm}).
\end{enumerate}
\end{lemma}
The adequate homogeneity of the maps $\Psi_{s,p}$ gives special properties to the corresponding $\mathscr{E}_{s,p}$. In particular, the functionals $\mathscr{E}_{s,p}$ are \textit{affine} , in the sense that the following holds.
\begin{prop}\label{affine invariance}
    Let $s>0$ and $1 \leq p<\infty$. For each $f \in \dot{W}^{s,p}$ and each $T \in \text{GL}_N$ such that $\verti{\text{det} \ T}=1$, we have
    \begin{flalign*}
        \mathscr{E}_{s,p}(f\circ T)= \mathscr{E}_{s,p}(f).
    \end{flalign*}
\end{prop}
This fact is a consequence of the following.
\begin{lemma}\label{change of var}
    Let $T \in \text{GL}_N$. Consider the map $F_T \colon \R^N \setminus \{0\} \to \R^N \setminus \{0\}$, where
    \begin{flalign}\label{ft}
        F_T(x) \coloneq \frac{T(x)}{\verti{T(x)}}, \ \fo x \in \R^N \setminus \{0\}.
    \end{flalign}
    For each measurable $g\colon \mathbb{S}^{N-1}\rightarrow \R^{+}$ , we have
    \begin{flalign*}
        \int_{\mathbb{S}^{N-1}}g(F_T(\omega))\ \frac{\verti{\text{det}\ T}}{\verti{T(\omega)}^{N}}\ d \mathscr{H}^{N-1}(\omega)= \int_{\mathbb{S}^{N-1}} g(\nu)\ d\mathscr{H}^{N-1}(\nu).
    \end{flalign*}
\end{lemma}
\begin{coro}\label{T W}
    Let $s$ be non-integer and $1 \leq p < \infty$. If $f \in \dot{W}^{s,p}$ and $T \in \text{GL}_N$, then $f\circ T \in W^{s,p}$.
\end{coro}

\begin{proof}[Proof of Lemma \ref{change of var}] We present a short proof of this result well-known to experts.
 On the one hand, we have
\begin{flalign*}
    \int_{\R^N}g\left(\frac{x}{\verti{x}}\right) \text{e}^{-\verti{x}^2} \ dx = \left(\int_{\R^N} \text{e}^{-r^2}r^{N-1}\ dr\right) \int_{\mathbb{S}^{N-1}}g(\nu) d \mathscr{H}^{N-1}(\nu).
\end{flalign*}    
On the other hand, we have
\begin{flalign*}\begin{split}
   \int_{\R^N}&g\left(\frac{x}{\verti{x}}\right) \text{e}^{-\verti{x}^2} \ dx \\ &= \int_{\R^N} g\left(\frac{T(y)}{\verti{T(y)}}\right) \text{e}^{-\verti{T(y)}^2} \verti{\text{det}\ T} \ dy
              \\ &=\int_{\mathbb{S}^{N-1}}g\left(\frac{T(\omega)}{\verti{T(\omega)}}\right) \left(\int_{0}^{\infty} \text{e}^{-\verti{T(\omega)}^2 r^2} r^{N-1} \ dr\right) \ \verti{\text{det} \ T} \ d\mathscr{H}^{N-1}(\omega)
              \\ & = \int_{\mathbb{S}^{N-1}}g\left(\frac{T(\omega)}{\verti{T(\omega)}}\right) \left(\int_{0}^{\infty} \text{e}^{-s^2} s^{N-1} \frac{ds}{\verti{T(\omega)}^{N}}\right) \verti{\text{det} \ T}d\mathscr{H}^{N-1}(\omega)
            \\ & = \left(\int_{0}^{\infty} \text{e}^{-s^2} s^{N-1}\ ds\right) \int_{\mathbb{S}^{N-1}} g\left(\frac{T(\omega)}{\verti{T(\omega)}}\right) \frac{\verti{\text{det}\ T}}{\verti{T(\omega)}^{N}}\ d\mathscr{H}^{N-1}(\omega).
            \end{split}
            \end{flalign*}
(Here, we make the changes of variables $x=T(y)$ to obtain the first equality and $s=\verti{T(\omega)}r$ to obtain the third one.)

\smallskip
We obtain the desired conclusion by comparing the two above formulas.         
\end{proof}

\begin{proof}[Proof of Proposition \ref{affine invariance}]
Assume, e.g., that $s$ is non-integer (the integer case is similar). 
 Let $f \in \dot{W}^{s,p}$, and $T \in \text{SL}_N$. 
We have
\bes
\ba
    &[\mathscr{E}_{s,p}(f\circ T)]^{-N/s}=\sigma_{N}^{-N/sp-1}\int_{\mathbb{S}^{N-1}} \left(\int_{0}^{\infty}u^{-sp-1}\vertii{\Delta_{u\frac{T(\xi)}{\verti{T(\xi)}}}^{\lfloor s \rfloor +1}f}_{L^p}^p \ du\right)^{-N/sp} \frac{1}{\verti{T(\xi)}^N}\ d\mathscr{H}^{N-1}(\xi)  
\ea
\ees
(here, we make the change of variables $u=\verti{T(\xi)}t$).

We obtain the desired conclusion by applying Lemma \ref{change of var} to the map
\bes
g(\omega)\coloneq \left(\int_{0}^{\infty}t^{-sp-1}\vertii{\Delta_{t\omega}^{\lfloor s \rfloor +1}f}_{L^p}^p \ dt\right)^{-N/sp}. \hfill \qedhere 
\ees
\end{proof}
We next state some auxiliary results that we will use in the proof of Theorem \ref{good fspace}. 
\begin{lemma}\label{oscar res}

Let $s$ be non-integer, $1 \leq p<\infty$, and $m$ be an integer $>s$. If $f \in L^1_{\text{loc}}$ satisfies 
    \begin{flalign*}
    \int_{\R^N}\frac{\vertii{\Delta_h^m f}_{L^p}^p}{ \verti{h}^{sp+N}} \ dh < \infty,
\end{flalign*}
    then there exists a polynomial $P$ such that 
    $f - P \in \dot{W}^{s,p}$.
\end{lemma}
\begin{lemma}\label{Zero Poly}
Let $m$ be an integer.
    If $P$ is a polynomial satisfying
    \begin{flalign*}\label{assumption lemma poly}
        \mathscr{H}^{N-1}\left(\left\{\xi \in \mathbb{S}^{N-1}; \, \exists \ t>0,  \  \vertii{\Delta^m_{t\xi}P}_{L^p}< \infty \right\}\right) >0,
    \end{flalign*} then $P$ is of degree $\leq m-1$, and therefore
    \begin{flalign*}
        \Delta_h^m P (x) = 0, \ \fo x \in \R^N, \ \fo h \in \R^N.
    \end{flalign*}
\end{lemma}

\begin{lemma}\label{int norm}
Let $s$ be an integer and $ 1 \leq p<\infty$.  Let $A \subset \mathbb{S}^{N-1}$ satisfy $\mathscr{H}^{N-1}(A)>0$. There exist $0<\mathcal{C}^1_{s,p,A}\leq \mathcal{C}^2_{s,p,A}<\infty$ such that, for each $f \in W^{s,1}_{\text{loc}}$, 
 \be \label{int} \mathcal{C}^1_{s,p,A}\verti{f}_{W^{s,p}}^p \leq \int_{A} \pari{\int_{\R^N}\verti{\partial^s_{\xi}f(x)}^p \ dx} \ d\mathscr{H}^{N-1}(\xi) \leq  \mathcal{C}^2_{s,p,A} \verti{f}_{W^{s,p}}^p.
    \ee 

\end{lemma} 

Granted the above results, we proceed to the proof of Theorem \ref{good fspace}.
\begin{proof}[Proof of Theorem \ref{good fspace}]
(1) Let $s$ be non-integer. If $f \in \dot{W}^{s,p}$, then $\mathscr{E}_{s,p}(f)<\infty$, by Lemma \ref{Jensen}. Conversely, if  $\mathscr{E}_{s,p}(f)< \infty$, then
\be\label{xi}
\mathscr{H}^{N-1}\left(\left\{ \xi \in \mathbb{S}^{N-1}; \, \int_{0}^{\infty} t^{-sp-1} \vertii{\Delta^{\lfloor s \rfloor +1}_{t\xi} f}^p_{L^p} \ dt < \infty \right\}\right)>0.
\ee

Consequently, there exists a basis $(u_1,\dots,u_N)$ of $\R^N$ such that, for each $1 \leq i \leq N$,
\bes
\int_{0}^{\infty} t^{-sp-1} \vertii{\Delta^{\lfloor s \rfloor +1}_{tu_i} f}^p_{L^p} \ dt <\infty.
\ees

Set $g\coloneq f\circ T^{-1}$, where $T$ is the linear transformation satisfying $T(u_i)=e_i$. 
For each $\xi \in \mathbb{S}^{N-1}$, we have
\bes
\ba
\int_{0}^{\infty} t^{-sp-1} \vertii{\Delta^{\lfloor s \rfloor +1}_{t\frac{T(\xi)}{\verti{T(\xi)}}} g}^p_{L^p} \ dt&= \frac{1}{\verti{T(\xi)}^{sp}}\int_{0}^{\infty} u^{-sp-1}\vertii{\Delta^{\lfloor s \rfloor +1}_{uT(\xi)} g}^p_{L^p} \ du \\ &= \frac{1}{\verti{T(\xi)}^{sp}} \int_{0}^{\infty} t^{-sp-1} \vertii{\pari{\Delta^{\lfloor s \rfloor +1}_{tT^{-1}(T(\xi))} f } \circ T^{-1}}^p_{L^p} \ dt  \\
& = \frac{\verti{\text{det} \ T}}{\verti{T(\xi)}^{sp}} \int_{0}^{\infty} t^{-sp-1} \vertii{\Delta^{\lfloor s \rfloor +1}_{t\xi} f}^p_{L^p} \ dt,
\ea
\ees
where we have used \eqref{delta comp} and performed the changes of variables $u=t\verti{T(\xi)}$, $y=T^{-1}(x)$.
Therefore, we have 
\bes
\begin{split}
 F_T\left(\left\{ \xi \in \mathbb{S}^{N-1};  \, \int_{0}^{\infty} t^{-sp-1} \vertii{\Delta^{\lfloor s \rfloor +1}_{t\xi} f}^p_{L^p} \ dt < \infty \right\} \right) \\ = \left\{ \xi \in \mathbb{S}^{N-1}; \, \int_{0}^{\infty} t^{-sp-1} \vertii{\Delta^{\lfloor s \rfloor +1}_{t\xi} g}^p_{L^p} \ dt < \infty \right\},
\end{split}
\ees
where $F_T$ is as in \eqref{ft}.
In particular, we have, for each $1 \leq i \leq N$,
\be \label{g slic}
\int_{0}^{\infty} t^{-sp-1} 
 \vertii{\Delta^{\lfloor s \rfloor +1}_{te_i} g}^p_{L^p} \ dt < \infty,
\ee
and
\be \label{g 2}
\mathscr{H}^{N-1}\left(\left\{ \xi \in \mathbb{S}^{N-1}; \, \int_{0}^{\infty} t^{-sp-1} \vertii{\Delta^{\lfloor s \rfloor +1}_{t\xi} g}^p_{L^p} \ dt < \infty \right\}\right) >0,
\ee
by \eqref{xi}.

\smallskip
If $s<1$, \eqref{g slic} and Theorem \ref{theo besov} imply that $g \in \dot{W}^{s,p}$ (recall that, in this case, \eqref{Besov} holds for each measurable function). Therefore $f \in \dot{W}^{s,p}$ (by Corollary \ref{T W}). Recall that, when $s<1$, 

\smallskip
If $s>1$, we argue as follows. 
 By \eqref{g slic} and Theorem \ref{higher slicing}, we have 
\be \label{hyos}
\int_{\R^N}\frac{\vertii{\Delta_h^{N(\lfloor s \rfloor +1)}g}^p_{L^p}} {\verti{h}^{sp+N}} \ dh < \infty.
\ee

By \eqref{hyos} and Lemma \ref{oscar res}, we find that there exists a polynomial $P$ such that 
$g-P \in \dot{W}^{s,p}$. By Theorem \ref{theo besov}, we have, for each $\xi \in \mathbb{S}^{N-1}$,
\be\label{P control}
(K^1_{s,p,N})^p\int_{0}^{\infty}t^{-sp-1}\vertii{\Delta^{\lfloor s \rfloor +1}_{t\xi}(g-P)}_{L^p}^p \ dt \leq \verti{g-P}_{W^{s,p}}^p.
\ee

Therefore, \eqref{g 2}, \eqref{P control}, and the triangular inequality imply that
\bes
 \mathscr{H}^{N-1}\left(\left\{ \xi \in \mathbb{S}^{N-1}; \, \int_{0}^{\infty}t^{-sp-1}\vertii{\Delta^{\lfloor s \rfloor +1}_{t\xi}P}_{L^p}^p \ dt < \infty\right\}\right)>0.
\ees

In particular, 
\bes
 \mathscr{H}^{N-1}\left(\left\{ \xi \in \mathbb{S}^{N-1}; \, \exists\, t>0\text{ such that } \vertii{\Delta^{\lfloor s \rfloor +1}_{t\xi}P}_{L^p}^p  < \infty\right\}\right)>0.
\ees

By Lemma \ref{Zero Poly}, we find that 
\bes
\Delta_h^{\lfloor s \rfloor +1} P (x) = 0, \ \fo x \in \R^N, \ \fo h \in \R^N.
\ees

Finally, 
\bes
\verti{g}_{W^{s,p}}=\verti{g-P}_{W^{s,p}}<\infty,
\ees
 which implies that $f \in \dot{W}^{s,p}$.
 
\smallskip 
\noindent (2) Let $s$ be an integer.

\smallskip
If $s=1$, we may argue  as in the proof of item (1) with $0<s<1$, using \eqref{slice s=1} instead of Theorem \ref{theo besov}. 

\smallskip
If $s\geq 2$, we argue as follows. Let $f \in W^{s,1}_{\text{loc}}$ be such that $\mathscr{E}_{s,p}(f)<\infty$. Then
\bes
\mathscr{H}^{N-1}\left(\left\{ \xi \in \mathbb{S}^{N-1}; \, \vertii{\partial^s_{\xi}f}_{L^p} < \infty \right\} \right)>0,
\ees
and therefore there exists $M<\infty$ such that 
\bes
\mathscr{H}^{N-1}\left(\left\{ \xi \in \mathbb{S}^{N-1}; \, \vertii{\partial^s_{\xi}f}_{L^p} < M \right\}\right)>0.
\ees

Set $A\coloneq \{ \xi \in \mathbb{S}^{N-1}; \, \vertii{\partial^s_{\xi}f}_{L^p} < M \}. $ By Lemma \ref{int norm} , there exists $C< \infty$ such that
\bes
\ba
\verti{f}_{W^{s,p}}^p \leq C \int_{A} \vertii{\partial^s_{\xi}f}_{L^p}^p  \ d\mathscr{H}^{N-1}(\xi)  \leq C M^p \mathscr{H}^{N-1}(A) < \infty,
\ea
\ees
which implies that $f \in \dot{W}^{s,p}$.
\end{proof}

We now turn to the proofs of Lemmas \ref{oscar res}, \ref{Zero Poly}, and \ref{int norm}. 
\begin{proof}[Proof of Lemma \ref{oscar res}]
This result is a direct consequence of the combination of Theorems $1$ and $3$ in \cite{dorronsoro1985mean}. 
\end{proof}

In the proof of Lemma \ref{Zero Poly}, we will rely on the following results.
\begin{lemma}\label{one t}
    Let $m$ be an integer and $P$ be a polynomial. If $\xi \in \R^N\setminus \{0\}$ is such that $\Delta^m_{\xi}P=0$, then $\partial^m_{\xi} P = 0$.    
\end{lemma}
\begin{lemma}\label{form norm}
Let $k$ be an integer,  $1 \leq p <\infty$, and $A \subset \mathbb{S}^{N-1}$ satisfy $\mathscr{H}^{N-1}(A)>0$. Then the map
    \bes
      \eta \mapsto \pari{\int_{A}\verti{\eta(\xi,\dots,\xi)}^p \ d\mathscr{H}^{N-1}(\xi)}^{1/p}  
    \ees
is a norm on the space of $k$-linear forms of $\R^N$.    
\end{lemma}
%\begin{proof}[Proof of Lemma \ref{one t}]
%Without loss of generality, we may assume that $P\neq 0$ and $\xi=(1,0,\dots,0)$. We may write $P$ as 
%\bes
%P(X_1,\dots,X_N)= \sum_{i=0}^d c_i(X_2,X_3,\dots,X_N) X_1^i,
%\ees
%where $d$ is possibly zero, and $c_d$ is a non-zero polynomial. Since 
%\bes
%\Delta^m_{e_1} P(x)= \sum_{i=0}^d c_i(x_2,\dots,x_N) \Delta^m_{e_1}(x_1^i)=0, \ \fo x \in \R^N, 
%\ees
%and, for each $0 \leq i\leq d$,
%\bes
%\deg(\Delta^m_{e_1}(X_1^i))=\begin{cases}
    %i-m,&\ \text{if} \  i\geq m \\
     %- \infty,& \ \text{else}
%\end{cases},
%\ees
%we find that $d<m$. Hence,  $\displaystyle \partial_{1}^m P(x)= \sum_{i=0}^d c_i(x_2,\dots,x_N) \partial_1^m(x_1^i) =0$, for each $x \in \R^N$.   
%\end{proof}
Lemma \ref{form norm} is a direct consequence of 
\begin{lemma}\label{p=0}
    Let $A \subset \mathbb{S}^{N-1}$ be such that $\mathscr{H}^{N-1}(A)>0$. If $P$ is a homogeneous polynomial satisfying $P(\xi)=0$, for each $\xi \in A$, then $P=0$.
\end{lemma}
\begin{proof}[Proof of Lemma \ref{p=0}]
Lemma \ref{p=0} follows from the homogeneity of $P$ and the following standard result: if $C \subset \R^{N}$ is such that $\verti{C}>0$, and if $P$ is a polynomial satisfying 
    $P(x)=0$, for each $x\in C$,
    then $P=0$.
\end{proof}
We now turn to the 
\begin{proof}[Proof of Lemma \ref{Zero Poly}]
Set $A\coloneq \left\{\xi \in \mathbb{S}^{N-1}; \,  \exists \  t>0,  \vertii{\Delta^m_{t\xi}P}_{L^p}< \infty \right\}$ . Let $\xi \in A$ and $t>0$ such that $\vertii{\Delta^m_{t\xi}P}_{L^p}<\infty$. Since the map
\bes
x \mapsto \Delta^{m} _{t\xi}P(x)
\ees
is a polynomial in $L^p$, we have $\Delta^{m}_{t\xi}P=0$, and therefore Lemma \ref{one t} yields
$\partial^m_{\xi}P=0$. 

\smallskip
Consequently, for each $x \in \R^N$, the map \ $\xi \mapsto \partial^m_{\xi} P(x)$
is a homogeneous polynomial vanishing on $A$. Combining this with Lemma \ref{p=0} we find that
\bes
\partial^m_{\xi}P(x)=0, \  \fo \xi \in \mathbb{S}^{N-1}, \ \fo x \in \R^N,
\ees
and thus $\deg (P)\le m-1$. This completes the proof of Lemma \ref{Zero Poly}.  
\end{proof}
\begin{proof}[Proof of Lemma \ref{int norm}]
 Let $f \in W^{s,1}_{\text{loc}}$. For a.e.\ $x \in \R^N$, the maps
 \bes
  (\R^N)^s \ni (h_1,\dots,h_s) \mapsto D_x^sf(h_1,\dots,h_s)
 \ees
 are $s$-linear forms. Therefore, by Lemma \ref{form norm}, there exist $0<\mathcal{C}^1_{s,p,A}\leq \mathcal{C}^2_{s,p,A} < \infty$ such that 
 \be \label{d ineq}
   \mathcal{C}^1_{s,p,A}\vertii{D^s_x f}^p \leq \int_{A}  \verti{\partial^s_{\xi} f (x)}^p \ d\mathscr{H}^{N-1}(\xi) \leq \mathcal{C}^2_{s,p,A} \vertii{D^s_x f}^p,
 \ee
 for a.e.\ $x \in \R^N$.
The conclusion of the lemma follows by integrating in $x$ \eqref{d ineq}. 
\end{proof}

\section{Applications of Theorems \ref{there is min intro} and \ref{constraint on min intro}}\label{appli}
In this short section, we present some straightforward consequences of Theorems \ref{there is min intro} and \ref{constraint on min intro}. The proof of Theorem \ref{constraint on min intro} will be given in Sections \ref{pf s=1} and \ref{sect general fractional}.
\begin{coro}\label{affine = classic}
Let $s>0$ and $1 \leq p<\infty$, with $p>1$ if $s$ is an integer $\geq 2$.
   For each $f \in \dot{W}^{s,p}$, there exists $T_f \in \text{SL}_N$ such that
   \begin{flalign}\label{equiv affine=cla}
    C^1_{s,p,N}\sigma_N^{1/p}\verti{ f\circ T_f}_{W^{s,p}}\leq \mathscr{E}_{s,p}(f \circ T_f) \leq C^2_{s,p,N}\sigma_{N}^{1/p}\verti{f \circ T_f}_{W^{s,p}},
   \end{flalign}
   where $C^1_{s,p,N}$ and $C^2_{s,p,N}$ are the constants given by Theorem \ref{constraint on min intro}. 
\end{coro}
 Recall that $\sigma_N= \mathscr{H}^{N-1}(\mathbb{S}^{N-1})$.
\begin{proof}[Proof of Corollary \ref{affine = classic}]
Let $f \in \dot{W}^{s,p}$. By Theorem \ref{there is min intro}, there exists $T_f \in \text{SL}_N$ such that
\bes
\verti{f \circ T_f}_{W^{s,p}}= \min \{ \verti{f \circ T}_{W^{s,p}}; \, T \in \text{SL}_N \}.
\ees

By  Theorem \ref{constraint on min intro}, we have
\be \label{tint 1}
\ba
\left(C^2_{s,p,N} \verti{f\circ T_f}_{W^{s,p}} \right)^{-N/s} &\leq \left(\int_{0}^{\infty}t^{-sp-1} \vertii{\Delta^{\lfloor s \rfloor +1}_{\xi} (f\circ T_f)}_{L^p}^p \ dt \right)^{-N/sp} \\ & \hspace{30 pt} \leq  \left(C^1_{s,p,N} \verti{f \circ T_f}_{W^{s,p}} \right)^{-N/s},
\ea
\ee
when $s$ is non-integer, respectively 
\be\label{tint 2}
\ba
\left(C^2_{s,p,N} \verti{f\circ T_f}_{W^{s,p}} \right)^{-N/s} &\leq \left(\int_{\R^N} \verti{\partial^s_{\xi}(f\circ T_f) (x)}^p \ dx \right)^{-N/sp}  \leq  \left(C^1_{s,p,N} \verti{f \circ T_f}_{W^{s,p}} \right)^{-N/s},
\ea
\ee
when $s$ is an integer. Corollary \ref{affine = classic} follows by integrating in $\xi$ \eqref{tint 1}, respectively   \eqref{tint 2}.
\end{proof}

We next derive Theorems \ref{subcrit aff emb} and \ref{gen embedding} from
Corollary \ref{affine = classic}.
 \begin{proof}[Proof of Theorem \ref{subcrit aff emb}]
     Let $s,p$ be such that $sp<N$, with $p>1$ if $s$ is an integer $\geq 2$. Let $f$ be in $\mathring{W}^{s,p}$. By Corollary \ref{affine = classic}, there exists $T_f \in \text{SL}_N$ such that 
\begin{flalign*}
    C^1_{s,p,N}\sigma_N^{1/p}\verti{f\circ T_f}_{W^{s,p}} \leq \mathscr{E}_{s,p}(f\circ T_f).
\end{flalign*}

On the other hand, 
the Sobolev inequality yields 
\begin{flalign*}
    \vertii{f\circ T_f}_{L^{Np/(N-sp)}}\leq \widetilde{C}_{s,p,N}\verti{f\circ T_f}_{W^{s,p}}
\end{flalign*}
for some finite constant $\widetilde{C}_{s,p,N}$.
Therefore, we have
\begin{flalign*}
    \vertii{f\circ T_f}_{L^{Np/(N-sp)}} \leq \frac{\widetilde{C}_{s,p,N}\sigma_N^{-1/p}}{C^1_{s,p,N}} \mathscr{E}_{s,p}(f\circ T_f).
\end{flalign*}

Since $\vertii{\cdot}_{L^{Np/(N-sp)}}$ and $\mathscr{E}_{s,p}$ are invariant under unimodular transformations (by Proposition \ref{affine invariance}) , the last inequality amounts to
\begin{flalign*}
    \vertii{f}_{L^{Np/(N-sp)}} \leq \frac{\widetilde{C}_{s,p,N}\sigma_N^{-1/p}}{C^1_{s,p,N}} \mathscr{E}_{s,p}(f).
\end{flalign*}
This completes the proof of Theorem \ref{subcrit aff emb}.
\end{proof}
In the proof of Theorem \ref{gen embedding}, we rely on the following optimal Sobolev embeddings (see \cite[Theorem 11.39]{leoni2023first}, \cite[Theorem B]{BREZIS20192839}, and Appendix \ref{B}).
\begin{theo}\label{Opt Sob} 
    Let $0<s_1<s_2<\infty$ and $1 \leq p_1,p_2 <\infty$ satisfy \eqref{Sob n}. There exists $\tilde{C}\coloneq \tilde{C}_{s_1,s_2,p_1,p_2,N}<\infty$ such that
    \bes
    \verti{f}_{W^{s_1,p_1}} \leq \tilde{C} \verti{f}_{W^{s_2,p_2}}, \ \fo f \in \dot{W}^{s_1,p_1} \cap \dot{W}^{s_2,p_2}.
    \ees
\end{theo}
\begin{proof}[Proof of Theorem \ref{gen embedding}]
Let $0<s_1<s_2<\infty$ and $1 \leq p_1,p_2<\infty$ satisfy \eqref{Sob n}
with $p_2>1$, if $s_2\geq 2$ is an integer.
Let $f \in \dot{W}^{s_1,p_1}\cap \dot{W}^{s_2,p_2}$. By Corollary \ref{affine = classic}, there exists $T_f \in \text{SL}_N$ such that
\be \label{g rep}
C^1_{s_2,p_2,N}\sigma_N^{1/p_2}\verti{f\circ T_f}_{W^{s_2,p_2}} \leq \mathscr{E}_{s_2,p_2}(f\circ T_f).
\ee
Theorem \ref{Opt Sob} yields 
\be \label{opt ineq}
\verti{f \circ T_f}_{W^{s_1,p_1}}\leq \tilde{C} \verti{f \circ T_f}_{W^{s_2,p_2}}.
\ee
On the other hand, by Lemma \ref{Jensen} and Proposition \ref{affine invariance}, we have 
\be \label{aff small}
\mathscr{E}_{s_1,p_1}(f)=\mathscr{E}_{s_1,p_1}(f\circ T_f) \leq {\alpha}_{s_1,p_1,N} \verti{f \circ T_f}_{W^{s_1,p_1}},
\ee
for some finite constant $\alpha_{s_1,p_1,N}$.
Combining \eqref{g rep}, \eqref{opt ineq} and \eqref{aff small}, we find that
\begin{align*}
\mathscr{E}_{s_1,p_1}(f) \leq \alpha_{s_1,p_1,N} \verti{f \circ T_f}_{W^{s_1,p_1}} &\leq \tilde{C}\alpha_{s_1,p_1,N}\verti{f\circ T_f}_{W^{s_2,p_2}} 
\\ & \leq \frac{\tilde{C}\alpha_{s_1,p_1,N}\sigma_N^{-1/p_2}}{C^1_{s_2,p_2,N}} \mathscr{E}_{s_2,p_2}(f).  \qedhere
\end{align*}
\end{proof}
\section{Proof of Theorem \ref{constraint on min intro} when \texorpdfstring{$s=1$}{s=1}}\label{pf s=1}

In this section, we present two proofs of Theorem \ref{constraint on min intro} in the case where $s=1$.  Our first approach yields Theorem \ref{constraint on min intro} with the constant
\begin{flalign}\label{C1p}
    C^1_{1,p,N}\coloneq \sup\left\{\frac{1/N- \lambda^{-\frac{1}{N-1}}}{\lambda - \lambda^{-\frac{1}{N-1}}}; \, \lambda>N^{N-1} \right\}.
\end{flalign} 
The second approach leads to a different constant
\begin{flalign}\label{cons pvar}
\widetilde{C}^1_{1,p,N}\coloneq \begin{cases} \displaystyle N^{-1/2}, & \text{if} \ p\geq 2,  \\
         \displaystyle N^{-1/p}, &  \text{if} \ 1 \leq p<2. 
             \end{cases}
    \end{flalign} 

See Remarks \ref{wek} and \ref{p=2} for further comments on $C^1_{1,p,N}$ and $\widetilde{C}^1_{1,p,N}$.    

\smallskip
We now turn to the proofs.
\begin{proof}[First proof of Theorem \ref{constraint on min intro} in the case where $s=1$]
    It suffices to prove that if $f \in \dot{W}^{1,p}$ is such that there exists $\xi \in \mathbb{S}^{N-1}$ satisfying
    \begin{flalign*}
         \left(\int_{\R^N} \verti{\nabla f(x) \cdot \xi}^p \ dx \right)^{1/p} < C^1_{1,p,N} \vertii{\nabla f}_{L^p},
    \end{flalign*}
    then there exists a transformation $T \in \text{SL}_N$ such that
\begin{flalign}\label{decreased norm}
    \vertii{\nabla(f\circ T)}_{L^p} < \vertii{\nabla f}_{L^p}.
\end{flalign}

Without loss of generality, we may assume that $\xi=(1,0,\dots,0)$ and thus
\begin{flalign}\label{assumption rotated}
  \left(\int_{\R^N}\verti{\partial_1 f(x)}^p \ dx\right)^{1/p}  < C^1_{1,p,N}\vertii{\nabla f}_{L^p}.
\end{flalign}
 
By \eqref{C1p} and \eqref{assumption rotated}, we may find $\lambda>N^{N-1}$ such that
\begin{flalign*}
    \left(\int_{\R^N}\verti{\partial_1 f(x)}^p \ dx\right)^{1/p} < \frac{\frac{1}{N} - \lambda^{-\frac{1}{N-1}}}{\lambda-\lambda^{-\frac{1}{N-1}}} \vertii{\nabla f}_{L^p}.
\end{flalign*}

We have
\begin{flalign*}
    \left(\int_{\R^N}\verti{\partial_1 f(x)}^p \ dx\right)^{1/p} &< \frac{\frac{1}{N} - \lambda^{-\frac{1}{N-1}}}{\lambda-\lambda^{-\frac{1}{N-1}}} \vertii{\nabla f}_{L^p}
    \\ & < \frac{\frac{1}{N} - \lambda^{-\frac{1}{N-1}}}{\lambda-\lambda^{-\frac{1}{N-1}}} \sum_{i=1}^N \left( \int_{\R^N} \verti{\partial_i f(x)}^p \ dx  \right)^{1/p}.
\end{flalign*}

Setting $\mu \coloneq \lambda^{-\frac{1}{N-1}}$ and multiplying the last inequality by $ \lambda- \mu$, we find that
\begin{flalign*}
    (\lambda - \mu )\left(\int_{\R^N}\verti{\partial_1 f(x)}^p \ dx \right)^{1/p} < \left(\frac{1}{N}- \mu\right) \sum_{i=1}^N \left(\int_{\R^N}\verti{\partial_i f(x)}^p \ dx \right)^{1/p},
\end{flalign*}
and, therefore,
\begin{flalign}\begin{split}\label{claim form 1}
    \lambda \left(\int_{\R^N} \verti{\partial_1 f(x)}^p \ dx\right)^{1/p} &+ \mu \sum_{i=2}^N \left(\int_{\R^N}\verti{\partial_i f(x)}^p \ dx \right)^{1/p}
   \\ & < \frac{1}{N} \sum_{i=1}^N \left(\int_{\R^N}\verti{\partial_i f(x)}^p \ dx \right)^{1/p}. 
  \end{split} 
\end{flalign}

Consider now the linear transformation
\begin{flalign*}
    T_{\lambda}\colon (x_1,\dots,x_N) \mapsto(\lambda x_1,\mu x_2, \dots, \mu x_n),
\end{flalign*}
which satisfies $ \text{det} \ T_{\lambda}=1$,
\begin{flalign*}
    \int_{\R^N} \verti{\partial_1(f\circ T_{\lambda})(x)}^p \ dx = \lambda^p \int_{\R^N} \verti{\partial_1 f(x)}^p \ dx,
\end{flalign*} 
and
\begin{flalign*}
    \int_{\R^N} \verti{\partial_{i}(f\circ T_{\lambda})(x)}^p\ dx = \mu^p \int_{\R^N}\verti{\partial_i f(x)}^p \ dx, \   \fo  2 \leq i \leq N.
\end{flalign*}
Hence, \eqref{claim form 1} reads as
\begin{flalign*}
    \sum_{i=1}^N \left(\int_{\R^N}\verti{\partial_i (f \circ T_{\lambda})(x)}^p \right)^{1/p} < \frac{1}{N} \sum_{i=1}^N \left(\int_{\R^N}\verti{\partial_i f(x)}^p \ dx \right)^{1/p}.
\end{flalign*}

Therefore, using \eqref{slice s=1}, we find that
\begin{flalign*}
    \vertii{\nabla (f\circ T_{\lambda})}_{L^p}&\leq \sum_{i=1}^N \left(\int_{\R^N}\verti{\partial_i (f \circ T_{\lambda})(x)}^p \ dx \right)^{1/p}\\ & < \frac{1}{N} \sum_{i=1}^N \left(\int_{\R^N}\verti{\partial_i f(x)}^p \ dx \right)^{1/p} \leq \vertii{\nabla f}_{L^p}. 
\end{flalign*}
This implies \eqref{decreased norm} for $T=T_{\lambda}$ and completes the proof of Theorem \ref{constraint on min intro} when $s=1$.
\end{proof}
 The approach presented in the proof above also yields the following.
\begin{prop}\label{more gen con on min}
    Let $1 \leq p< \infty$ and $\gamma>1$. There exists $C(\gamma)>0$ such that if $f \in \dot{W}^{1,p}$ satisfies 
    \begin{flalign*}
        \vertii{\nabla f}_{L^p}\leq \gamma \min\{\vertii{\nabla (f\circ T)}_{L^p};\, \ T \in \text{SL}_N\},
    \end{flalign*}
    then
    \begin{flalign*}
        C(\gamma)\vertii{\nabla f}_{L^p}\leq \left(\int_{\R^N}\verti{\nabla f(x) \cdot \xi}^p \ dx\right)^{1/p}.
    \end{flalign*}
\end{prop} 
 The second proof of Theorem \ref{constraint on min intro} when $s=1$ relies on the following fact. 
\begin{lemma}\label{diff}
    Let $1 \leq p < \infty$ and consider $g \in L^p(\R^N; \R^N)$. The map
    \begin{flalign}\label{def SL fun}
        \Psi\colon \text{GL}_N \ni L \mapsto \int_{\R^N} \verti{L g(x)}^p \ dx
    \end{flalign}
    is differentiable.  Its differential at $L_0 \in \text{GL}_N$ is the linear form given by 
    \begin{flalign}\label{diff formula}
        D_{L_0}\Psi (M) = p\int_{\R^N}\mathbbm{1}_{[g \neq 0]} \verti{ L_0  g(x)}^{p-2} \left( L_0 g(x) \cdot M g(x) \right) \ dx,
    \end{flalign}
    for each $M \in \text{M}_N$.
\end{lemma}
Lemma \ref{diff} applied to  $g\coloneq\nabla f$, $f \in \dot{W}^{1,p}$, and the chain rule, imply the following.
\begin{coro}\label{coro dif}
    Let $1 \leq p<\infty$ and consider $f \in \dot{W}^{1,p}$. The map 
   \bes
     \widetilde{\Psi} \colon \text{GL}_N \ni L \mapsto \int_{\R^N} \verti{ L^{\text{T}} \nabla f(x) }^p \ dx 
    \ees
    is differentiable at $L_0 \in \text{GL}_N$. Its differential is the linear form given by
    \bes
      D_{L_0}\widetilde{\Psi} (M) = p\int_{\R^N}\mathbbm{1}_{[\nabla f \neq 0]} \verti{  L_0^{\text{T}}  \nabla f (x)}^{p-2} \left( L_0^{\text{T}} \nabla f(x) \cdot  M^{\text{T}} \nabla f(x) \right) \ dx,
    \ees
    for each $M \in \text{M}_N$.
\end{coro}
Granted Lemma \ref{diff}, we turn to the
\begin{proof}[Second proof of Theorem \ref{constraint on min intro} in the case where $s=1$]
   Let $f \in \dot{W}^{1,p}$ be  such that 
   \begin{flalign*}
    \vertii{\nabla f}_{L^p}= \min \{ \vertii{\nabla (f\circ T)}_{L^p}; \, \ T \in \text{SL}_N \}.
    \end{flalign*}
It suffices to show that 
\begin{flalign}\label{redest}
    \widetilde{C}^1_{1,p,N}\vertii{\nabla f}_{L^p}\leq \vertii{\partial_1 f}_{L^p}.
\end{flalign}

Consider the map
\bes
\Psi \colon \text{GL}_N \ni L \mapsto \vertii{\nabla f}^p_{L^p}= \int_{\R^N}\verti{L^{\text{T}} \nabla f(x)}^p \ dx.
\ees

The restriction of $\Psi$ to $\text{SL}_N$, still denoted $\Psi$ for simplicity, reaches its minimum at $I_N$. Therefore we have $(D_{I_N}\Psi)_{|T_{I_N}\text{SL}_N}=0$, where $T_{I_N}\text{SL}_N= \{ M \in \text{M}_N; \ \text{tr}(M)=0 \}$ is the tangent space to $\text{SL}_N$ at $I_N$. Therefore, by Corollary \ref{coro dif}, we have  
\begin{flalign}\label{crit point}
    \int_{\R^N}\mathbbm{1}_{\nabla f\neq 0} \verti{\nabla f}^{p-2} \nabla f \cdot M^{T}\nabla f \ dx=0, \ \text{for each} \  M \ \text{such that} \  \text{tr}(M)=0.
\end{flalign}

Letting, in \eqref{crit point}, $M \coloneq I_N-\text{diag}(N,0,\dots,0)$, we find that
\begin{flalign}\label{equi formula}
    \frac{1}{N} \int_{\R^N} \verti{\nabla f}^p \ dx = \int_{\R^N} \verti{\nabla f}^{p-2}\verti{\partial_1 f}^2 \ dx.
\end{flalign} 

 If $p\geq 2$, an application of H\"older's inequality shows that
    \begin{flalign*}
        \frac{1}{N}\int_{\R^N}\verti{\nabla f}^p \ dx= \int_{\R^N}\verti{\nabla f}^{p-2}\verti{\partial_1 f}^2 \ dx \leq \vertii{\nabla f}_{L^p}^{p-2} \vertii{\partial_1 f}_{L^p}^{2}.
    \end{flalign*}   
    We conclude that
    \begin{flalign*}
        \frac{1}{\sqrt{N}}\vertii{\nabla f}_{L^p} \leq \vertii{\partial_1 f}_{L^p}.
    \end{flalign*}
    
If $1 \leq p<2$, we have 
    \begin{flalign*}
        \int_{\R^N}\verti{\nabla f}^{p-2}\verti{\partial_1 f}^{2} \ dx \leq \int_{\R^N} \verti{\partial_1 f}^{p} \ dx. 
    \end{flalign*} 
    This fact combined with \eqref{equi formula} yields
    \bes
        \frac{1}{N^{1/p}} \vertii{\nabla f}_{L^p} \leq \vertii{\partial_1 f}_{L^p}. \qedhere
    \ees
\end{proof}

\smallskip
We now turn to the proof of Lemma \ref{diff}. When 
 $p>1$, Lemma \ref{diff} is a consequence of the following well-known result, combined with the chain rule.
\begin{lemma}\label{Lp diff}
    Let $1<p<\infty$. The map 
    \bes
      G \colon L^p(\R^N ; \R^N) \ni g \mapsto \int_{\R^N}\verti{g(x)}^p \ dx
    \ees
    is $C^1$ and its differential at $g_0 \in L^p(\R^N ; \R^N)$ is given by
    \bes
     D_{g_0}G(h) = p \int_{\R^N} \verti{g_0(x)}^{p-2} \left(g_0(x) \cdot h(x) \right) \ dx,\ \forall\,  h \in L^p(\R^N ; \R^N).
    \ees
\end{lemma}

It remains to consider the case where $p=1$.
\begin{proof}[Proof of Lemma \ref{diff} in the case where $p=1$]
Let $g \in L^1(\R^N ; \R^N)$ and $L_0 \in \text{GL}_N$. 
The differentiability at $L_0$ of
\bes
\Psi \colon \text{GL}_N \ni L \mapsto \int_{\R^N} \verti{Lg(x)} \ dx
\ees
will follow from
\be \label{domconv}
\ba
\int_{\R^N}\frac{1}{\vertii{H_n}}\bigg|\verti{(L_0+H_n) g(x)} &-\verti{L_0 g(x)}\\ &- \mathbbm{1}_{[g \neq 0]} \verti{L_0g(x)}^{-1}\left(L_0 g(x) \cdot  H_n g(x) \right) \bigg| \  dx \longrightarrow 0, 
\ea
\ee
for each $(H_n) \subset \text{M}_N$ that converges to $0$, property that we now show. For this purpose, we argue as follows. For each $x \in \R^N$, the map 
\bes
\varphi_x \colon \text{GL}_N \ni L \mapsto \verti{L g(x)}
\ees
is differentiable. If $g(x)=0$, then $\varphi_x=0$. If $g(x) \neq 0$, the chain rule yields
\bes
D_{L_0}\varphi_x(H)= \verti{L_0g(x)}^{-1}\left(L_0g(x) \cdot Hg(x) \right), \ \fo H \in \text{M}_N.
\ees

This implies that the integrand in \eqref{domconv} converges to $0$  as $n \to \infty$. It remains to find a suitable domination.

\smallskip
 We have 
\bes
\ba
&\bigg|\verti{(L_0+H_n) g(x)} -\verti{L_0 g(x)}- \mathbbm{1}_{[g \neq 0]} \verti{L_0g(x)}^{-1}\left(L_0 g(x) \cdot H_n g(x) \right) \bigg|
\\ & \leq \bigg| \verti{(L_0+H_n) g(x)} -\verti{L_0 g(x)} \bigg| + \bigg| \mathbbm{1}_{[g \neq 0]} \verti{L_0g(x)}^{-1}\left(L_0 g(x) \cdot H_n g(x) \right) \bigg|
\\ & \leq 2 \vertii{H_n} \verti{g(x)},
\ea
\ees
for each $x \in \R^N$.
Hence,  
\bes
\ba
&\frac{1}{\vertii{H_n}}\bigg|\verti{(L_0+H_n) g(x)} -\verti{L_0 g(x)}- \mathbbm{1}_{[g \neq 0]} \verti{L_0g(x)}^{-1}\left(L_0 g(x) \cdot H_n g(x) \right) \bigg| \\ & \leq 2 \verti{g(x)}, 
\ea
\ees
for each $x \in \R^N$ and $n$. 
\eqref{domconv} then follows by dominated convergence and this completes the proof of Lemma \ref{diff}.
\end{proof}
\begin{rema}
Identity \eqref{equi formula} appears in \cite[Theorem 1.2]{huang2018characterization}.
\end{rema}
\begin{rema}\label{wek}
We note that our second approach yields a sharper bound than our first one. More specifically, we have
\begin{flalign*}
    \frac{C^1_{1,p,N}}{\widetilde{C}^1_{1,p,N}} \rightarrow 0 \ \text{as} \ N \rightarrow \infty,
\end{flalign*}
for every $1 <p < \infty$, by \eqref{cons pvar} and
since $C^1_{1,p,N} \leq N^{-1}$.
\end{rema}
\begin{rema}\label{p=2}
 In the special case where $p=2$, we have, by \eqref{equi formula},
\begin{flalign}
    \left(\int_{\R^N}\verti{\nabla f(x) \cdot \xi}^2 \ dx\right)^{1/2} = \frac{1}{\sqrt{N}}\vertii{\nabla f}_{L^2}, \fo \xi \in \mathbb{S}^{N-1},
\end{flalign}
for each $f \in \dot{W}^{1,2}$ such that
\bes
\vertii{\nabla f}_{L^2}=\min \{ \vertii{\nabla(f \circ T)}_{L^2}; \, T \in \text{SL}_N \}.
\ees

Hence, we find that, for such a function $f$, 
\begin{flalign}\label{eq sob}
    \mathscr{E}_{1,2}(f)=\sqrt{\frac{\sigma_N}{N}}\vertii{\nabla f}_{L^2}.
\end{flalign}

It is straightforward that \eqref{eq sob} implies the optimal affine Sobolev inequality in \cite[Theorem 1]{lutwak2002sharp}, when $N \geq 3$ and $p=2$.
\end{rema}
\section{Proof of Theorem \ref{constraint on min intro} in the general case}\label{sect general fractional}
In this section, we prove that Theorem \ref{constraint on min intro} holds with 
\be \label{def C}
\ba
    C^1_{s,p,N}&\coloneq \sup   
    \left\{\frac{K^1_{s,p,N}-K^2_{s,p,N} \lambda^{-\frac{s}{N-1}}}{(K^2_{s,p,N})^2(\lambda^s - \lambda^{-\frac{s}{N-1}})}; \,  \lambda > \left(\frac{K^2_{s,p,N}}{K^1_{s,p,N}}\right)^{(N-1)/s}  \right\}, \\
    C^2_{s,p,N}&\coloneq \frac{1}{K^1_{s,p,N}},
\ea
\ee
where $K^1_{s,p,N}$ and $K^2_{s,p,N}$ are the constants given by Theorems \ref{theo besov} and \ref{integer slicing}. 

\smallskip
It is straightforward that $C^1_{s,p,N}>0$. 
We refer to Section \ref{cons} for further remarks on $C^1_{s,p,N}$.

\smallskip
\begin{proof}[Proof of Theorem \ref{constraint on min intro}] 

(1) If $s$ is non-integer, we argue as follows. It suffices to prove that if $f \in \dot{W}^{s,p}$ is such that there exists $\xi \in \mathbb{S}^{N-1}$ satisfying
    \begin{flalign}\label{contrapo hypo}
    \left(\int_{0}^{\infty}t^{-sp-1}\vertii{\Delta^{\lfloor s \rfloor +1}_{t\xi }f}_{L^p}^p dt \right )^{1/p}  < C^1_{s,p,N}\verti{f}_{W^{s,p}},
    \end{flalign}
then there exists a unimodular transformation $T \in \text{SL}_N$ such that  
\begin{flalign}\label{final claim frac}
\verti{f\circ T}_{W^{s,p}} < \verti{f}_{W^{s,p}}.
\end{flalign}
Without loss of generality, we may assume that $\xi=(1,0,\dots,0)$ and thus
\be\label{e1}
 \left(\int_{0}^{\infty}t^{-sp-1}\vertii{\Delta^{\lfloor s \rfloor +1}_{te_1 }f}_{L^p}^p dt \right )^{1/p}  < C^1_{s,p,N}\verti{f}_{W^{s,p}}.
\ee
Using \eqref{def C} and \eqref{e1}, we obtain the existence of some $\displaystyle \lambda > \left(\frac{K^2_{s,p,N}}{K_{s,p,N}^1}\right)^{(N-1)/s}$ such that, with $\mu\coloneq \lambda^{-1/(N-1)}$, 
\begin{flalign}\label{by def of C}
\begin{split}
    &\left(\int_{0}^{\infty}t^{-sp-1}\vertii{\Delta^{\lfloor s \rfloor +1}_{te_1 }f}_{L^p}^p \ dt  \right)^{1/p} < \frac{K^1_{s,p,N}-K^2_{s,p,N} \mu^s}{(K^2_{s,p,N})^2\left(\lambda^s - \mu^s\right)}  \verti{f}_{W^{s,p}} \\& \hspace{50 pt} \leq \frac{K^1_{s,p,N}-K^2_{s,p,N} \mu^s}{K^2_{s,p,N}\left(\lambda^s - \mu^s\right)}  \sum_{i=1}^N  \left(\int_{0}^{\infty}t^{-sp-1}\vertii{\Delta^{\lfloor s \rfloor +1}_{te_i }f}_{L^p}^p \ dt\right)^{1/p}.
    \end{split}
\end{flalign}
(For the last inequality, we use Theorem \ref{theo besov}.) 

\smallskip
Therefore, 
\begin{flalign*}
& K^{2}_{s,p,N}(\lambda^s-\mu^s)\left(\int_{0}^{\infty}t^{-sp-1}\vertii{\Delta^{\lfloor s \rfloor +1}_{te_1 }f}_{L^p}^p \ dt  \right)^{1/p} \\ & \hspace{50 pt} <  \left(K^1_{s,p,N}-K^2_{s,p,N}\mu^s \right)\sum_{i=1}^N  \left(\int_{0}^{\infty}t^{-sp-1}\vertii{\Delta^{\lfloor s \rfloor +1}_{te_i }f}_{L^p}^p dt \right)^{1/p},
\end{flalign*}
and, thus,
\begin{flalign}\label{des ineq}
\begin{split}
& K^2_{s,p,N}\bigg[ \lambda^s \left(\int_{0}^{\infty}t^{-sp-1}\vertii{\Delta^{\lfloor s \rfloor +1}_{te_1 }f}_{L^p}^p \ dt\right)^{1/p} \\ &   \hspace{40 pt} + \mu^s \sum_{i=2}^N \left(\int_{0}^{\infty}t^{-sp-1}\vertii{\Delta^{\lfloor s \rfloor +1}_{te_i }f}_{L^p}^p \ dt\right)^{1/p} \bigg] \\ & \hspace{60 pt}< K^1_{s,p,N}  \sum_{i=1}^N  \left(\int_{0}^{\infty}t^{-sp-1}\vertii{\Delta^{\lfloor s \rfloor +1}_{te_i }f}_{L^p}^p dt \right)^{1/p}.
\end{split}
\end{flalign} 
Consider now the linear transformation
\begin{flalign*}
    T_{\lambda}\colon \R^N \ni (x_1,\dots,x_N)\mapsto (\lambda x_1, \mu x_2, \dots,\mu x_N),
\end{flalign*}
which satisfies $\text{det} \  T_{\lambda} =1.$ By Lemma \ref{frac chan}, we have 
\be \label{der 1}
    \int_{0}^{\infty}t^{-sp-1}\vertii{\Delta^{\lfloor s \rfloor +1}_{te_1}(f\circ T_{\lambda})}_{L^p}^p \ dt= \lambda^{sp}\int_{0}^{\infty}t^{-sp-1}\vertii{\Delta^{\lfloor s \rfloor +1}_{t e_1}f}_{L^p}^p \ dt,
\ee
and, for each $2 \leq i \leq N$,
\be \label{der 2}
    \int_{0}^{\infty}t^{-sp-1}\vertii{\Delta^{\lfloor s \rfloor +1}_{te_i}(f\circ T_{\lambda})}_{L^p}^p \ dt=\mu^{sp}\int_{0}^{\infty}t^{-sp-1}\vertii{\Delta^{\lfloor s \rfloor +1}_{te_i}f}_{L^p}^p \ dt.
\ee
Hence, \eqref{des ineq} reads
\begin{flalign*}
\begin{split}
    K^2_{s,p,N} &\sum_{i=1}^N \left(\int_{0}^{\infty}t^{-sp-1}\vertii{\Delta^{\lfloor s \rfloor +1}_{te_i}(f\circ T_{\lambda})}_{L^p}^p \ dt \right)^{1/p}   
    \\ & \hspace{30 pt} < K^1_{s,p,N} \sum_{i=1}^N \left(\int_{0}^{\infty}t^{-sp-1}\vertii{\Delta^{\lfloor s \rfloor +1}_{te_i}f}_{L^p}^p \ dt \right)^{1/p}.
\end{split}    
\end{flalign*}
Therefore, by Theorem \ref{theo besov}, we find that
\begin{flalign*}
\begin{split}
    \verti{f \circ T_{\lambda}}_{W^{s,p}} & \leq K^2_{s,p,N} \sum_{i=1}^N \left(\int_{0}^{\infty}t^{-sp-1}\vertii{\Delta^{\lfloor s \rfloor +1}_{te_i}(f\circ T_{\lambda})}_{L^p}^p \ dt \right)^{1/p}   
    \\ &  < K^1_{s,p,N} \sum_{i=1}^N \left(\int_{0}^{\infty}t^{-sp-1}\vertii{\Delta^{\lfloor s \rfloor +1}_{te_i}f}_{L^p}^p \ dt \right)^{1/p} \leq \verti{f}_{W^{s,p}}.
\end{split}
\end{flalign*}

This completes the proof of Theorem \ref{constraint on min intro} in the case where $s$ is non-integer. 

\smallskip
\noindent
(2) The proof of item (2) is essentially the same as the above one. The following modifications are required. 

\begin{enumerate}[(a)]
    \item Instead of \eqref{der 1} and \eqref{der 2}, we use the identities 
    \bes
     \ba
    \left(\int_{\R^N}\verti{\partial^s_1(f\circ T_{\lambda})(x)}^p \ dx \right)^{1/p} &= \lambda^s \left(\int_{\R^N}\verti{\partial^s_1 f (x)}^p \ dx \right)^{1/p}, \\
\left(\int_{\R^N}\verti{\partial^s_i(f\circ T_{\lambda})(x)}^p \ dx \right)^{1/p} &= \mu^s \left(\int_{\R^N}\verti{\partial^s_i f (x)}^p \ dx \right)^{1/p}, \ \fo 2 \leq i \leq N.
     \ea
    \ees
\item In place of Theorem \ref{theo besov}, we rely on Theorem \ref{integer slicing}. \qedhere
\end{enumerate}
\end{proof}

The proof of Theorem \ref{constraint on min intro} also yields the following analogue of Proposition \ref{more gen con on min}.
\begin{prop}
 \begin{enumerate}[(1)]
     \item  Let $s$ be non integer, $1\leq p<\infty$ and $\gamma\geq 1$. There exists $C(\gamma)>0$ such that, if $f \in \dot{W}^{s,p}$ satisfies
 \begin{flalign*}
    \verti{f}_{W^{s,p}}\leq \gamma \min\{\verti{f\circ T}_{W^{s,p}};\, \ T \in \text{SL}_N \},
 \end{flalign*}
 then
 \begin{flalign*}
   C(\gamma)\verti{f}_{W^{s,p}} \leq \left(\int_{0}^{\infty}t^{-sp-1}\vertii{\Delta_{t\xi}^{\lfloor s \rfloor +1}f}_{L^p}^p \ dt \right)^{1/p}, \ \fo \xi \in \mathbb{S}^{N-1}.
 \end{flalign*}
 \item Let $s\geq 2$ be an integer, $1<p<\infty$, and $\gamma \geq 1$. There exists $C(\gamma)>0$ such that if $f \in \dot{W}^{s,p}$ satisfies
 \begin{flalign*}
    \verti{f}_{W^{s,p}}\leq \gamma \min\{\verti{f\circ T}_{W^{s,p}};\, \ T \in \text{SL}_N \},
 \end{flalign*}
 then
 \begin{flalign*}
  C(\gamma)\verti{f}_{W^{s,p}} \leq  \left(\int_{\R^N} \verti{\partial^s_{\xi} f(x)}^p dx\right)^{1/p}, \ \fo \xi \in \mathbb{S}^{N-1}.
    \end{flalign*}
\end{enumerate}
\end{prop}
In the same vein, we note the following result.
\begin{prop}\label{lap}
    Let $\gamma \geq 1$. There exists $C(\gamma)>0$ such that, if $f \in \dot{W}^{2,1}$ satisfies
\be\label{min lap}
\vertii{\Delta f}_{L^1} \leq \gamma \inf \left\{ \vertii{\Delta(f \circ T)}_{L^1}; \, T \in \text{SL}_N  \right\},
\ee    
then 
\bes
 C(\gamma)\vertii{\Delta f}_{L^1} \leq \int_{\R^N} \verti{\partial^2_{\xi} f(x)} \  dx,  \ \fo \xi \in \mathbb{S}^{N-1}.
\ees
\end{prop}
\noindent
(In the above setting, we do not claim the existence of a minimizer in the right-hand side of \eqref{min lap}.)

\begin{proof}[Proof of Proposition \ref{lap}]
    Let $\gamma \geq 1$. We prove that the conclusion holds with
    \be\label{Cgamma}
    C(\gamma) \coloneq \sup \left\{ \frac{1-\gamma\lambda^{-2/(N-1)}}{\gamma(\lambda^2+\lambda^{-2/(N-1)})}; \, \lambda> \gamma^{(N-1)/2} \right\}.
    \ee
    
    We argue as in the proof of Theorem \ref{constraint on min intro}. It suffices to prove that, if $f \in \dot{W}^{2,1}$ is such that there exists $\xi \in \mathbb{S}^{N-1}$ satisfying
    \bes
    \vertii{\partial^2_{\xi} f}_{L^1}< C(\gamma)\vertii{\Delta f}_{L^1},
    \ees
    then there exists $T \in \text{SL}_N$ such that
    \be\label{claim lap}
     \gamma \vertii{\Delta(f\circ T)}_{L^1}< \vertii{\Delta f}_{L^1},
    \ee
    which is the desired contradiction.

    \smallskip
    Since the Laplace operator commutes with isometries, we may assume that $\xi=(1,0,\dots,0)$, and thus
    \be\label{hyp}
     \vertii{\partial^2_{1} f}_{L^1}< C(\gamma)\vertii{\Delta f}_{L^1}.
    \ee
    
    By \eqref{Cgamma} and \eqref{hyp}, there exists some $\lambda> \gamma^{(N-1)/2}$ such that
    \be \label{hyp lam}
    \vertii{\partial_1^{2} f}_{L^1} < \frac{1-\gamma\lambda^{-2/(N-1)}}{\gamma(\lambda^2+\lambda^{-2/(N-1)})} \vertii{\Delta f}_{L^1}.
    \ee

    Set $\mu \coloneq \lambda^{-1/(N-1)}$ and consider $T_{\lambda}$ as in the proof of Theorem \ref{constraint on min intro}. 
    We have
    \be\label{T est}
    \ba
     \vertii{\Delta(f\circ T_{\lambda})}_{L^1} &\leq \vertii{\partial_1^2 (f \circ T_{\lambda})}_{L^1}+ \int_{\R^N}\verti{\sum_{i=2}^N \partial_i^2(f\circ T_{\lambda})(x)} \ dx  
     \\ &= \lambda^2 \vertii{\partial^2_1 f}_{L^1} + \mu^2 \int_{\R^N}\verti{\sum_{i=2}^N \partial_i^2 f(x)} \ dx.
    \ea
    \ee
    
    On the other hand, we have
    \be\label{rest}
     \int_{\R^N}\verti{\sum_{i=2}^N \partial_i^2 f(x)} \ dx \leq \vertii{\Delta f}_{L^1}+ \vertii{\partial_1^2 f}_{L^1}.
    \ee
    
    Combining \eqref{hyp lam}, \eqref{T est}, and \eqref{rest}, we find that
    \bes
        \vertii{\Delta(f\circ T_{\lambda})}_{L^1} < \left( \frac{1-\gamma\mu^2}{\gamma(\lambda^2+\mu^2)}(\lambda^2+\mu^2) + \mu^2 \right) \vertii{\Delta f}_{L^1}= \frac{1}{\gamma} \vertii{\Delta f}_{L^1}.
    \ees
    
    Hence, \eqref{claim lap} holds with $T_{\lambda}$ and this completes the proof of Proposition \ref{lap}.
\end{proof}

Proposition \ref{lap} implies the following \enquote{weak} affine Sobolev inequality, which complements Theorem \ref{subcrit aff emb} in the borderline case where  $s=2$ and $p=1$.
\begin{theo}\label{weak}
    Assume that $N\geq 3$. There exists $K_{N}<\infty$ such that
    \be\label{we}
    \vertii{f}_{L^{N/(N-2),\infty}}\leq K_N \mathscr{E}_{2,1}(f), \ \fo f \in C_c^{\infty},
    \ee
    where $L^{N/(N-2), \infty}$ is the weak Lebesgue space, equipped with
    \be
    \label{faible}
     \vertii{f}_{L^{N/(N-2),\infty}}\coloneq \sup_{t>0} \ t \ \verti{\left\{x \in \R^N; \, \verti{f(x)}>t \right\}}^{(N-2)/N}.
    \ee
\end{theo}
\begin{rema}
    Note that, by Markov's inequality, 
     $\vertii{f}_{L^{N/(N-2),\infty}} \leq \vertii{f}_{L^{N/(N-2)}}$, for each measurable $f$. This explains why we refer to $L^{N/(N-2),\infty}$ as a \enquote{weak} Lebesgue space and to \eqref{we} as a \enquote{weak} affine Sobolev inequality.
\end{rema}
We rely on the following (see Zygmund \cite[p.247]{Zygaund1989}, Ponce \cite[Proposition 5.7]{ponce2016elliptic}).
\begin{theo}\label{cal}

Assume that $N\geq 3$. There exists $ K_N<\infty$ such that
\bes
 \vertii{f}_{L^{N/(N-2),\infty}}\leq K_N \vertii{\Delta f}_{L^1}, \ \fo f \in C_c^{\infty}.
\ees
\end{theo}
\begin{proof}[Proof of Theorem \ref{weak}] In what follows, $C$ denotes a general constant that depends only on $N\geq 3$.

\smallskip
We argue as in the proof of Theorem \ref{h order aff sob}.  Let $f \in C^{\infty}_c$. Let $T_f \in \text{SL}_N$ such that
\bes
\vertii{\Delta (f\circ T_f)}_{L^1}\leq 2 \inf\left\{\vertii{\Delta(f \circ T)}_{L^1}; \, T \in \text{SL}_N \right\}.
\ees

By Proposition \ref{lap}, we have
    \bes
    \vertii{\Delta(f\circ T_f)}_{L^1}\leq C \int_{\R^N} \verti{\partial_{\xi}^2 (f\circ T_f)(x)} \ dx, \ \fo \xi \in \mathbb{S}^{N-1},
    \ees
    and this yields
    \be\label{D est}
     \vertii{\Delta(f\circ T_f)}_{L^1} \leq C \mathscr{E}_{2,1}(f \circ T_f)=\mathscr{E}_{2,1}(f),
    \ee
    using Proposition \ref{affine invariance}.

    \smallskip
    On the other hand, an inspection of \eqref{faible} shows that
    $\vertii{f \circ T_f}_{L^{N/(N-2),\infty}}=\vertii{f}_{L^{N/(N-2),\infty}}$. Combining \eqref{D est} and Theorem \ref{cal}, we find that
    \bes
     \vertii{f}_{L^{N/(N-2),\infty}}=\vertii{f \circ T_f}_{L^{N/(N-2),\infty}}\leq C \vertii{\Delta(f\circ T_f)}_{L^1}\leq C \mathscr{E}_{2,1}(f).\qedhere
    \ees
\end{proof}
\section{A closer look at the case where \texorpdfstring{$0<s<1$}{0<s<1}}\label{cons}
In this section, we make a quantitative comparison between our approach to affine Sobolev inequalities and the one developed in \cite{haddad2024affine}, when $0<s<1$. 

\smallskip
The proof of \eqref{h order aff sob} in \cite{haddad2024affine} goes as follows. Let $f \in W^{s,p}$ and $f^{\#}$ be the symmetric decreasing rearrangement of $f$.
Clearly, we have 
\be\label{pro}
\vertii{f}_{L^q}= \vertii{f^{\#}}_{L^q}, \ 
 \mathscr{E}_{s,p}(f^{\#})=  \verti{f^{\#}}_{W^{s,p}}, \ \text{and} \ \vertii{f^{\#}}_{L^q} \leq \widetilde{C}_{s,p,N}\verti{f^{\#}}_{W^{s,p}},
 \ee
where the second equality follows from \eqref{rad} and \eqref{sp}, and  $\widetilde{C}_{s,p,N}$ is the best Sobolev constant.
One of the main contributions of \cite{haddad2024affine} consists in establishing the affine P\'olya-Szeg\"o inequality
\be \label{pza}
\mathscr{E}_{s,p}(f^{\#}) \leq \mathscr{E}_{s,p}(f).
\ee

\eqref{pro} and \eqref{pza} obviously imply \eqref{lud had} with
\be\label{sobc}
C_{s,p,N}\coloneq  \widetilde{C}_{s,p,N}.
\ee

Moreover, the above considerations show that  we have equality in \eqref{lud had} if $f$ is an extremizer in the Sobolev inequality, and therefore $C_{s,p,N}$ is the best constant.

\smallskip
By contrast, our approach relies on the fact that
\bes
\vertii{f \circ T}_{L^q}=\vertii{f}_{L^q}, \ \vertii{ f \circ T}_{L^q}\leq \widetilde{C}_{s,p,N} \verti{f \circ T}_{W^{s,p}}, \ \fo T \in \text{SL}_N,
\ees
and on the existence of $T_f \in \text{SL}_N$ such that
\bes
\verti{f \circ T_f}_{W^{s,p}} \leq \frac{\sigma_N^{-1/p}}{C^1_{s,p,N}} \mathscr{E}_{s,p}(f \circ T_f)= \frac{\sigma_N^{-1/p}}{C^1_{s,p,N}} \mathscr{E}_{s,p}(f).
\ees

Here, $C^1_{s,p,N}$ is the constant in Theorem \ref{constraint on min intro}. This yields \eqref{lud had}, with the constant
\be\label{k}
K_{s,p,N}\coloneq \frac{\widetilde{C}_{s,p,N} \sigma_N^{-1/p}}{C_{s,p,N}^1}, 
\ee
instead of the optimal constant $C_{s,p,N}$ given by formula \eqref{sobc}. 

\smallskip
Although there is no hope to expect that $K_{s,p,N}=C_{s,p,N}$ in general, we observe that the estimate we obtain is \enquote{not much worse} than \eqref{lud had}. To be more precise,   there exists $C\coloneq C_{N}<\infty$ such that  
    \be\label{compar}
     K_{s,p,N}\leq C C_{s,p,N}, 
    \ee
    for each $0<s<1$ and $1 \leq p<\infty$.

Indeed, Remark \ref{co s 1} and the proof of Theorem \ref{constraint on min intro} imply that Theorem \ref{constraint on min intro} holds with a positive constant $C^1_{s,p,N}=C^1_N$ that only depends on $N$. By \eqref{sobc} and \eqref{k}, this implies \eqref{compar}.
%\begin{rema}
    %Theorem \ref{oc} is based on the "good behaviour" of the constants $K^1_{s,p,N},K^2_{s,p,N}$ when $0<s<1$. However, such a phenomenon does not occur when $s>1$. Indeed, if we consider the best constants $\tilde{K}^2_{s,p,N}$ such that the right-hand inequality in \eqref{Besov} holds, and an integer $m \geq 2$, then we have
    %\bes
     %\tilde{K}^2_{s,1,N} \to \infty \ \ \text{as} \ s \to m^{-}.
    %\ees
%\end{rema}
\section{Gagliardo--Nirenberg type inequalities}\label{gag sec}
This section is devoted to the proof of Theorem \ref{aff gag}.
It is based on the following, see, e.g., \cite[Theorems 7.50, 11.42]{leoni2023first} and \cite{BREZIS20181355}. 
\begin{theo}\label{gag}
    Let $0<s_1<s_2<\infty$, $1 < p_1,p_2 < \infty $, and $\theta \in (0,1)$. Set $s\coloneq \theta s_2 +(1-\theta)s_1$ and $1/p \coloneq \theta/p_2 +(1-\theta)/p_1$.  There exists $\widetilde{C}\coloneq \widetilde{C}_{s_1,p_1,s_2,p_2,\theta,N}<\infty$ such that
    \bes
    \verti{f}_{W^{s,p}}\leq \widetilde{C} \verti{f}_{W^{s_1,p_1}}^{1-\theta}\verti{f}_{W^{s_2,p_2}}^{\theta}, \ \fo f \in \dot{W}^{s_1,p_1}\cap \dot{W}^{s_2,p_2}.
    \ees
     Same when $0<s_1<s_2\leq 1$ and $1 \leq p_1,p_2<\infty$, with $s_1p_1<1$ if $s_2=1$ and $p_2=1$.
\end{theo}
\begin{proof}[Proof of Theorem \ref{aff gag}]
    In what follows,  $C$ denotes a general positive constant that only depends on $s_1,p_1,s_2,p_2,s,p,$ and $N$.  
    By Theorems \ref{there is min intro} and \ref{constraint on min intro}, it suffices to show that
    \be
      C \mathscr{E}_{s,p}(f) \leq  \verti{f \circ T_1}_{W^{s_1,p_1}}^{1-\theta} \verti{f \circ T_2}_{W^{s_2,p_2}}^{\theta},
      \ee
    for each $T_1, T_2  \in \text{SL}_N$, and $f \in \dot{W}^{s_1,p_1} \cap \dot{W}^{s_2,p_2}$. This amounts to 
    \be \label{red gag}
     C\mathscr{E}_{s,p}(f) \leq  \verti{f \circ T}_{W^{s_1,p_1}}^{1-\theta} \verti{f}_{W^{s_2,p_2}}^{\theta},
      \ee
    for each $T \in \text{SL}_N$ and $f \in \dot{W}^{s_1,p_1} \cap \dot{W}^{s_2,p_2}$

\smallskip 
 We claim that it actually suffices to prove that 
\be \label{diag red}
 C \mathscr{E}_{s,p}(f) \leq  \verti{f \circ (\text{DO})}_{W^{s_1,p_1}}^{1-\theta} \verti{f}_{W^{s_2,p_2}}^{\theta},
\ee
for each diagonal $\text{D} \in \text{SL}_N$, $\text{O} \in \text{O}_N$, and $f \in \dot{W}^{s_1,p_1} \cap \dot{W}^{s_2,p_2}$.
Indeed, each $T \in \text{SL}_N$ can be written as  $T=\text{O}\text{S}$, with $\text{O} \in \text{O}_N$ and a symmetric matrix $\text{S} \in \text{SL}_N$. The spectral theorem yields a matrix $\tilde{\text{O}} \in \text{O}_N$ and a diagonal matrix $\text{D} \in \text{SL}_N$ such that $T=\text{O}\tilde{\text{O}}^{\text{T}}\text{D} \tilde{\text{O}}$. This implies that \eqref{red gag} holds for $f$ and $T$ if and only if
\bes
C \mathscr{E}_{s,p}(g) \leq  \verti{g\circ (\text{D} \text{O})
}_{W^{s_1,p_1}}^{1-\theta} \verti{g}_{W^{s_2,p_2}}^{\theta}, 
\ees
where $g\coloneq f \circ(\text{O} \tilde{\text{O}}^{\text{T}}).$ This proves our claim.

\smallskip
We now prove \eqref{diag red}. Let $f \in \dot{W}^{s_1,p_1}\cap\dot{W}^{s_2,p_2}$, $\text{O} \in \text{O}_N$, and $\text{D}=\text{diag}(\lambda_1,\lambda_2,\dots,\lambda_N) \in \text{SL}_N$. Consider the orthonormal basis $u_i \coloneq \text{O}^{-1}e_i$. If $s_1$ is non-integer, by Theorem \ref{theo besov}, Lemma \ref{frac chan}, and Remark \ref{1d norm}, we have
\be\label{frac sli}
\begin{split}
\verti{f \circ (\text{DO})}_{W^{s_1,p_1}}\geq C \left(\int_{0}^{\infty}t^{-s_1p_1-1} \vertii{\Delta_{tu_i}^{\lfloor s_1 \rfloor +1}(f\circ(\text{DO}))}_{L^{p_1}}^{p_1} \ dt \right)^{1/p_1}
\\ \geq C \verti{\lambda_i}^{s_1}\left(\int_{\R^{N-1}}\verti{f(x_1,\dots,x_{i-1},\cdot,x_{i+1},\dots,x_N)}_{W^{s_1,p_1}(\R)}^{p_1}  d \widehat{x_i}\right)^{1/p_1}, 
\end{split}
\ee
for each $1 \leq i \leq N$. 

\smallskip
Similarly, if $s_1$ is an integer, by Theorem \ref{integer slicing} and Remark \ref{1d norm}, we have
\begin{flalign}\label{int sli}
\begin{split}
&\verti{f \circ (\text{DO})}_{W^{s_1,p_1}}\geq C \vertii{\partial^{s_1}_{u_i}(f\circ (\text{DO}))}_{L^{p_1}}
\\ & \hspace{20 pt} \geq C \verti{\lambda_i}^{s_1}\left(\int_{\R^{N-1}}\verti{f(x_1,\dots,x_{i-1},\cdot,x_{i+1},\dots,x_N)}_{W^{s_1,p_1}(\R)}^{p_1}  d \widehat{x_i}\right)^{1/p_1}, 
\end{split}
\end{flalign}
for each $1 \leq i \leq N$.

\smallskip
Using \eqref{frac sli}, respectively  \eqref{int sli},  we obtain
\be\label{1 sli}
\ba
& \verti{f \circ (\text{DO})}_{W^{s_1,p_1}}^{(1-\theta)p} \\ &\geq C \sum_{i=1}^N \verti{\lambda_i}^{s_1(1-\theta)p}\left(\int_{\R^{N-1}}\verti{f(x_1,\dots,x_{i-1},\cdot,x_{i+1},\dots,x_N)}^{p_1}_{W^{s_1,p_1}(\R)}  \ d \widehat{x_i}\right)^{(1-\theta)p/p_1}.
\ea
\ee

Similarly, we have
\be\label{2 sli}
\verti{f}_{W^{s_2,p_2}}^{\theta p} \geq C\sum_{i=1}^N \left(\int_{\R^{N-1}} \verti{f(x_1,\dots,x_{i-1},\cdot,x_{i+1},\dots,x_N)}_{W^{s_2,p_2}(\R)}^{p_2}  d \widehat{x_i}\right)^{\theta p/p_2}.  
\ee

Therefore, by \eqref{1 sli} and \eqref{2 sli}, we have
    \begin{flalign*}
       & \verti{f \circ (\text{DO})}_{W^{s_1,p_1}}^{(1-\theta)p} \verti{f}_{W^{s_2,p_2}}^{\theta p} &&
       \\ & \geq C \sum_{i=1}^N \verti{\lambda_i}^{s_1(1-\theta)p}\left(\int_{\R^{N-1}} \verti{f(x_1, \dots,x_{i-1}, \cdot, x_{i+1},\dots,x_N)}^{p_1}_{W^{s_1,p_1}(\R)} \ d \widehat{x_i}\right)^{(1-\theta)p/p_1}  \\ &  \hspace{40 pt} \times \sum_{\ell=1}^N \left(\int_{\R^{N-1}} \verti{f(x_1, \dots,x_{\ell-1}, \cdot, x_{\ell+1},\dots,x_N)}^{p_2}_{W^{s_2,p_2}(\R)} \ d \widehat{x_{\ell}}\right)^{\theta p/p_2} 
       \\ & \geq C  \sum_{i=1}^N \bigg[\verti{\lambda_i}^{s_1(1-\theta)p}\left(\int_{\R^{N-1}} \verti{f(x_1, \dots,x_{i-1}, \cdot, x_{i+1},\dots,x_N)}^{p_1}_{W^{s_1,p_1}(\R)}\ d \widehat{x_i} \right)^{(1-\theta)p/p_1} \\ &  \hspace{30 pt }\times \left(\int_{\R^{N-1}} \verti{f(x_1, \dots,x_{i-1}, \cdot, x_{i+1},\dots,x_N)}^{p_2}_{W^{s_2,p_2}(\R)} \ d \widehat{x_i}\right)^{\theta p/p_2} \bigg] 
     \\ &  \geq C \sum_{i=1}^N \verti{\lambda_i}^{s_1(1-\theta)p} \int_{\R^{N-1}} \bigg(\verti{f(x_1, \dots,x_{i-1}, \cdot, x_{i+1},\dots,x_N)}^{p(1-\theta)}_{W^{s_1,p_1}(\R)} \\ &  \hspace{130 pt}\times \verti{f(x_1, \dots,x_{i-1}, \cdot, x_{i+1},\dots)}^{p\theta}_{W^{s_2,p_2}(\R)}\bigg) \ d\widehat{x_i},
    \end{flalign*}
    where the last inequality follows from H\"older's inequality. Applying Theorem \ref{gag} to the functions $f(x_1,\dots,x_{i-1},\cdot,x_{i+1},\dots,x_N)$, we find that
    \be \label{for g}
    \ba
    \verti{f \circ (\text{DO})}_{W^{s_1,p_1}}^{(1-\theta)p} &\verti{f}_{W^{s_2,p_2}}^{\theta p} \\ &\geq C \sum_{i=1}^N \verti{\lambda_i}^{s_1(1-\theta)p} \int_{\R^{N-1}} \verti{f(x_1,\dots,x_{i-1},\cdot,x_{i+1},\dots,x_N)}_{W^{s,p}(\R)}^p \ d \widehat{x_i}.
    \ea
    \ee
    Consider now $\widetilde{\text{D}}\coloneq \text{diag}(\verti{\lambda_1}^{s_1(1-\theta)/s},\dots, \verti{\lambda_N}^{s_1(1-\theta)/s}) \in \text{SL}_N$ and the function $g \coloneq f \circ \widetilde{\text{D}}$. We have
    \bes
    \ba
     \int_{\R^{N-1}} &\verti{g(x_1,\dots,x_{i-1},\cdot ,x_{i+1},\dots,x_N)}_{W^{s,p}(\R)}^p \ 
     d\widehat{x_i} \\ & = \verti{\lambda_i}^{s_1(1-\theta)p} \int_{\R^{N-1}}  \verti{f(x_1,\dots,x_{i-1},\cdot,x_{i+1},\dots,x_N)}_{W^{s,p}(\R)}^p \ 
     d\widehat{x_i}, 
    \ea
    \ees
    for each $1 \leq i \leq N,$ by Lemma \ref{frac chan} and Remark \ref{1d norm}. Therefore, \eqref{for g} reads
    \bes
    \ba
        &\verti{f \circ (\text{DO})}_{W^{s_1,p_1}}^{(1-\theta)p} \verti{f}_{W^{s_2,p_2}}^{\theta p} \\ &\geq C\sum_{i=1}^N  \int_{\R^{N-1}} \verti{g(x_1,\dots,x_{i-1},\cdot,x_{i+1},\dots,x_N)}_{W^{s,p}(\R)}^p \ 
     d\widehat{x_i}
    \ea
    \ees
    and Theorem \ref{theo besov} (if $s$ is non-integer), respectively Theorem \ref{integer slicing} (if $s$ is an integer), yield
    \be\label{final}
     \verti{f \circ (\text{DO})}_{W^{s_1,p_1}}^{(1-\theta)p} \verti{f}_{W^{s_2,p_2}}^{\theta p} \geq C \verti{g}_{W^{s,p}}^p.
    \ee
    
    On the other hand, we have
 $\verti{g}_{W^{s,p}}^p \geq  \alpha_{s,p,N}\mathscr{E}_{s,p}(g)^p $
    and $\mathscr{E}_{s,p}(f)=\mathscr{E}_{s,p}(g)$ (by Lemma \ref{Jensen} and Proposition \ref{affine invariance}). Hence, \eqref{final} yields
    \bes
     \verti{f \circ (\text{DO})}_{W^{s_1,p_1}}^{(1-\theta)p} \verti{f}_{W^{s_2,p_2}}^{\theta p} \geq C \mathscr{E}_{s,p}(f)^p.
    \ees
    
  This completes the proof of Theorem \ref{aff gag}.
\end{proof}    
\section{Reverse affine inequalities}\label{rev ineq}
In this section, we prove Theorems \ref{start reverse}, \ref{gen rev affine ineq}, and \ref{cara rev}.

\smallskip
Clearly, Theorem \ref{gen rev affine ineq} follows from Theorems \ref{start reverse}, \ref{there is min intro}, and \ref{constraint on min intro}. The proof of Theorem \ref{start reverse} relies on the following.
\begin{lemma}\label{lem rev}
 Let $R>0$ and $1 \leq p<\infty$.
 \begin{enumerate}[(1)]
\item Let $s$ be non-integer. There exists $C_{s,p,R}<\infty$ such that, for each $f \in W^{s,p}$ supported in $B(0,R)$ (see Appendix \ref{A} for the definition of $W^{s,p}$),  \bes
\vertii{f}_{L^p}^p\leq C_{s,p,R} \left(\int_{0}^{\infty} t^{-sp-1}\vertii{\Delta^{\lfloor s \rfloor +1}_{t\xi} f}_{L^p}^p \ dt \right), \ \fo \xi \in \mathbb{S}^{N-1}. 
\ees
\item Let $s$ be an integer.  There exists $C_{s,p,R}<\infty$ such that, for each $f \in W^{s,p}$ supported in $B(0,R)$,
\bes
\vertii{f}_{L^p}^p\leq C_{s,p,R}\vertii{\partial^s_{\xi} f}_{L^p}^p, \ \fo \xi \in \mathbb{S}^{N-1}. 
\ees
\end{enumerate}
\end{lemma}
Lemma \ref{lem rev} was established (with explicit constants) in \cite{haddad2021affine} when $s=1$ and in \cite{dominguez2024new} when $0<s<1$. In full generality, it is a consequence of the following Poincaré inequality.
\begin{lemma}\label{1d Poi}
    Assume that $N\geq 1$. Let $R>0$, $s>0$, and $1 \leq p<\infty$. There exists $C_{s,p,R,N}<\infty$ such that
    \bes
    \vertii{f}_{L^p}^p \leq C_{s,p,R,N} \verti{f}_{W^{s,p}}^p, 
    \ees
    for each $f \in W^{s,p}$ supported in $B(0,R)$. 
\end{lemma}
\begin{proof}[Proof of Lemma \ref{lem rev}]
  Let $f \in W^{s,p}$ be supported in $B(0,R)$. 
  
  \smallskip
  
  If $s$ is non-integer, we argue as follows.  It suffices to prove that
    \bes
     \vertii{f}_{L^p}^p \leq C_{s,p,R} \int_{0}^{\infty}t^{-sp-1}\vertii{\Delta_{te_1}^{\lfloor s \rfloor +1}f}_{L^p}^p \ dt, 
    \ees
    for some finite $C_{s,p,R}$.
    By Remark \ref{1d norm}, we have
    \bes
    2 \int_{0}^{\infty}t^{-sp-1} \vertii{\Delta_{te_1}^{\lfloor s \rfloor +1}f}_{L^p}^p \ dt = \int_{\R^{N-1}} \verti{f(\cdot, x_2,\dots,x_N)}_{W^{s,p}(\R)}^p \ d \widehat{x_1}.
    \ees
    Therefore, by Lemma \ref{1d Poi}, we have
    \be
    \ba
      \vertii{f}_{L^p}^p&=\int_{\R^{N-1}} \vertii{f(\cdot,x_2,\dots,x_N)}_{L^p(\R)}^p \ d \widehat{x_1} \\ &\leq C_{s,p,R}\int_{\R^{N-1}} \verti{f(\cdot, x_2,\dots,x_N)}_{W^{s,p}(\R)}^p \ d \widehat{x_1}\\ &= 2 C_{s,p,R} \int_{0}^{\infty}t^{-sp-1} \vertii{\Delta_{te_1}^{\lfloor s \rfloor +1}f}_{L^p}^p \ dt,
    \ea
    \ee
    and this completes the proof of Lemma \ref{lem rev} in the case where $s$ is non-integer.

    \smallskip
    The integer case follows from similar arguments.
\end{proof}
We now turn to the 
\begin{proof}[Proof of Theorem \ref{start reverse}]
    Let $R>0$, $s>0$, and $1 \leq p<\infty$, with $p>1$ if $s\geq 2$ is an integer. Let $f \in W^{s,p}$ be supported in $B(0,R)$ and $T \in \text{SL}_N$. 

    \smallskip
    We may assume, arguing as in the proof of Theorem \ref{aff gag}, that $T=\text{DO}$,  where 
    \ $\text{D}=\text{diag}(\lambda_1,\dots,\lambda_N) \in \text{SL}_N$ and $\text{O} \in \text{O}_N$.
    Let $(u_1,\dots,u_N)$ be the orthonormal basis of $\R^N$ defined by $\text{O}u_i=e_i$.

    \smallskip
    If $s$ is non-integer, we argue as follows. By Theorem \ref{theo besov} and Lemma \ref{frac chan}, we have
    \be \label{1frac}
    \ba
     \verti{f \circ (\text{DO})}^p_{W^{s,p}}&\geq C\sum_{i=1}^N \int_{0}^{\infty}t^{-sp-1}\vertii{\Delta^{\lfloor s \rfloor +1}_{t u_i} (f\circ(\text{DO}))}_{L^p}^p \ dt  
     \\ & = C\sum_{i=1}^N \verti{\lambda_i}^{sp}\int_{0}^{\infty}t^{-sp-1}\vertii{\Delta^{\lfloor s \rfloor +1}_{t e_i} f}_{L^p}^p \ dt .
    \ea
    \ee
By the AM-GM inequality and since $\displaystyle \prod_{i=1}^N \lambda_i = 1$, we have
\begin{flalign}\label{ag}
\begin{split}
 \frac{1}{N}&\left(\sum_{i=1}^N \verti{\lambda_i}^{sp}\int_{0}^{\infty}t^{-sp-1}\vertii{\Delta^{\lfloor s \rfloor +1}_{t e_i} f}_{L^p}^p \ dt \right) \\ & \hspace{20 pt} \geq \left(\prod_{i=1}^N \verti{\lambda_i}^{sp} \int_{0}^{\infty}t^{-sp-1}\vertii{\Delta^{\lfloor s \rfloor +1}_{t e_i} f}_{L^p}^p \ dt \right)^{1/N}
 \\  & \hspace{35 pt }\geq \left(\prod_{i=1}^N \int_{0}^{\infty}t^{-sp-1}\vertii{\Delta^{\lfloor s \rfloor +1}_{t e_i} f}_{L^p}^p \ dt \right)^{1/N}.
\end{split}
\end{flalign}
On the other hand, by Lemma \ref{lem rev},
\be\label{a i}
C \vertii{f}_{L^p}^p \leq \int_{0}^{\infty}t^{-sp-1}\vertii{\Delta^{\lfloor s \rfloor +1}_{t e_i} f}_{L^p}^p \ dt, \ \fo 1 \leq i\leq N, 
\ee 
and the right-hand inequality in \eqref{Besov} implies that there exists $1 \leq j \leq N$ such that
\be \label{e j}
C \verti{f}_{W^{s,p}}^p \leq \int_{0}^{\infty}t^{-sp-1}\vertii{\Delta^{\lfloor s \rfloor +1}_{t e_j} f}_{L^p}^p \ dt.
\ee
Combining \eqref{ag}, \eqref{a i}, and \eqref{e j}, we find that
\be\label{sum e}
\sum_{i=1}^N \verti{\lambda_i}^{sp}\int_{0}^{\infty}t^{-sp-1}\vertii{\Delta^{\lfloor s \rfloor +1}_{t e_i} f}_{L^p}^p \ dt \geq C\vertii{f}_{L^p}^{p(1-1/N)} \verti{f}_{W^{s,p}}^{p/N}.
\ee
Using \eqref{1frac} and \eqref{sum e}, we obtain
\bes
\verti{f \circ (\text{DO})}^p_{W^{s,p}} \geq C \vertii{f}_{L^p}^{p(1-1/N)}\verti{f}_{W^{s,p}}^{p/N}
\ees
and this completes the proof of Lemma \ref{start reverse} when $s$ is non-integer.

\smallskip
In the case where $s$ is an integer, we may argue similarly, using the identities
\bes
\vertii{\partial_{u_i}^s (f\circ \text{DO})}^p_{L^p}= \verti{\lambda_i}^{sp} \vertii{\partial_{i}^s f}_{L^p}^p, \ \fo 1 \leq i \leq N,
\ees
and relying on Theorem \ref{integer slicing} instead of Theorem \ref{theo besov}.
\end{proof}

\begin{proof}[Proof of Theorem \ref{cara rev}]
Without loss of generality, we assume that $R=1$.

\smallskip
\noindent
(1) We obtain the \enquote{if} part of (1) relying on Theorem \ref{gen rev affine ineq}, Lemma \ref{1d Poi}, and the fact that $\vertii{f}_{L^q}  \leq C\vertii{f}_{L^p} $, for each $f \in W^{1,p}$ supported in $B(0,1)$.

     The \enquote{only if} part of (1) is implicit in \cite[Proof of Theorem 2]{haddad2021affine}.  When $p>1$ (the case $p=1$ is included in (2)), following \cite{haddad2021affine}, we may consider the functions 
    \bes 
    f_k\colon x \mapsto \phi_k(x_1) \eta(x_2,\dots,x_N),  
    \ees
    where $\eta$ is a smooth function supported in $B_{N-1}(0,1/2)$, the ball of radius $1/2$ centered at $0$ in $\R^{N-1}$, and 
    \bes
     \phi_k(x_1) \coloneq \begin{cases}
         1+k/6 -k\verti{x_1-1/2},& \ \text{if} \ x_1 \in [1/3-1/k,1/3] \cup [2/3,2/3+1/k]\\
         1, & \ \text{if} \ x_1 \in [1/3,2/3]\\
         0,& \ \text{else}
     \end{cases}.
    \ees
    The $f_k$'s are supported in $B(0,1)$ for each sufficiently large integer $k $.
    In \cite[Proof of Theorem 2]{haddad2021affine}, it is shown that $\mathscr{E}_{1,p}(f) \leq C k^{(p-1)/pn}$ and $\vertii{\nabla f}_{L^p} \geq C k^{(p-1)/p}$ for each $k$. Thus, for any $\theta>1/N$, we have
    \bes
    \frac{\mathscr{E}_{1,p}(f_k)}{\vertii{\nabla f}_{L^p}^{\theta}} \to 0, \ \text{as}\  k \to \infty,
    \ees
    while $\inf_k \vertii{f_k}_{L^1}>0$, and therefore \eqref{ineq 1} fails if $\theta>1/N$.

    \smallskip
    \noindent
   (2) In order to prove the \enquote{only if} part of (2), we argue as follows. Let $0\leq \theta \leq 1$ be such that \eqref{ineq 2} holds. For each $f \in W^{1,p}$ supported in $B(0,1)$ and each $\lambda>0$, we have 
    \be \label{all lam}
    \ba
    \vertii{f}_{L^q}^{1-\theta} \vertii{\nabla f}_{L^p}^{\theta}&\leq C \mathscr{E}_{1,p}(f) \\ & = C \mathscr{E}_{1,p}(f\circ T_{\lambda}) \leq  C\left( \lambda\vertii{\partial_1 f}_{L^p} + \lambda^{-1/(N-1)}  \sum_{i=2}^N  \vertii{\partial_i f}_{L^p} \right), 
    \ea
    \ee
    where 
     \bes
     T_{\lambda}: (x_1,\dots,x_N) \mapsto (\lambda x_1, \lambda^{-1/(N-1)}x_2, \dots,\lambda^{-1/(N-1)}x_N).
    \ees
      Here, we rely on Proposition \ref{affine invariance} to obtain the equality (since $\text{det} \ T_{\lambda}=1)$, and on Lemma \ref{Jensen} and \eqref{slice s=1} for the second inequality.

    If   $\vertii{\partial_1 f}_{L^p}\neq 0$, applying \eqref{all lam} to $\displaystyle \lambda\coloneq  \left(\frac{\sum_{i=2}^N \vertii{\partial_i f}_{L^p}}{\vertii{\partial_1 f}_{L^p}} \right)^{1-1/N}$ yields
      \be \label{scale}
      \vertii{f}_{L^q}^{1-\theta} \vertii{\partial_1 f}_{L^p}^{\theta} \leq \vertii{f}_{L^q}^{1-\theta} \vertii{\nabla f}_{L^p}^{\theta} \leq C \vertii{\partial_1 f}_{L^p}^{1/N} \left( \sum_{i=2}^N \vertii{\partial_i f}_{L^p} \right)^{1-1/N}. 
      \ee
      
      Considering non-zero functions $\displaystyle \varphi \in C^{\infty}_c\big( (-1/2,1/2) \big) $ and $\displaystyle \psi \in C_c^{\infty}\big( B_{N-1}(0,1/2) \big) $, and applying \eqref{scale} to the maps
       \bes
        f_{\varepsilon}(x) \coloneq \varphi(x_1/\varepsilon) \psi(x_2,\dots,x_N), \ 0<\varepsilon<1,  
    \ees
    we find that $\varepsilon^{1/q+\theta(1/p-1/q-1)}  \leq C \varepsilon^{1/p-1/N}$, for each $0<\varepsilon<1$, and this yields 
    \bes
    \theta \leq \frac{1/N+1/q-1/p}{1+1/q-1/p}.
    \ees

We now prove the \enquote{if} part of (2) as follows. 

\smallskip
Let $ p\leq q <\infty$ be such that $\displaystyle\theta_{\text{max}}\coloneq\frac{1/N+1/q-1/p}{1+1/q-1/p}\geq 0$. This condition is equivalent to $q\leq Np/(N-p)$  when $p<N$, and always holds when $p\geq N$. Equivalently, 
\be
\label{ggf}
\theta_{\text{max}}\ge 0
\ \text{if and only if the embedding $W^{1,p}(B(0,1))\hookrightarrow L^q(B(0,1))$ holds}.
\ee

In view of \eqref{ggf}, it suffices to show that \eqref{ineq 2} holds with $\theta=\theta_{\text{max}}$. In turn,
 the proof of \eqref{ineq 2} with $\theta=\theta_{\text{max}}$ goes as follows. 
 Let $s\coloneq N/p-N/q$ which satisfies $q=Np/(N-sp)$.   If $p \geq N$, we have $ 0 \leq s \leq 1$. If $p<N$, this is also the case, since $q\leq Np/(N-p)$.  By Theorems \ref{subcrit aff emb} and \ref{aff gag}, we have
\be \label{gns}
\vertii{f}_{L^q} \leq C \mathscr{E}_{s,p}(f) \leq \vertii{f}_{L^p}^{1-s}\mathscr{E}_{1,p}(f)^{s},
\ee
for each $f \in W^{1,p}$. Hence, for each $f \in W^{1,p}$ supported in $B(0,1)$, we have
\be \label{last}
\ba
\vertii{f}_{L^q}^{1-1/N} & \vertii{\nabla f}_{L^p}^{1/N+1/q-1/p} \\ &\leq C \vertii{f}_{L^p}^{(1-s)(1-1/N)} \mathscr{E}_{1,p}(f)^{s(1-1/N)}\vertii{\nabla f}_{L^p}^{1/N+1/q-1/p}
\\ & = C (\vertii{f}_{L^p}^{(1-1/N)} \vertii{\nabla f}_{L^p}^{1/N})^{(1-s)} \mathscr{E}_{1,p}(f)^{s(1-1/N)} 
\\ & \leq C\mathscr{E}_{1,p}(f)^{(1-s)} \mathscr{E}_{1,p}(f)^{s(1-1/N)}
\\ & = C\mathscr{E}_{1,p}(f)^{1+1/q-1/p}.
\ea
\ee
Here, we rely on \eqref{gns} for the first inequality, on the definition of $s$ for the first and the second equality, and on Theorem \ref{gen rev affine ineq} for the last inequality. We obtain the desired conclusion, raising \eqref{last} to the power $\displaystyle \frac{1}{1+1/q-1/p}$.  
\end{proof}
\appendix
\section{From inhomogeneous to homogeneous slicing}\label{A}
For the sake of completeness, we explain in Appendices \ref{A} and \ref{B} how to obtain homogeneous slicing (Theorem \ref{integer slicing}) and Sobolev embeddings (Theorem \ref{Opt Sob}) from their inhomogeneous counterparts.

\smallskip
In both cases, a first step consists in proving homogeneous inequalities for $C_c^{\infty}$ functions, using their inhomogeneous counterparts. This easily follows from a scaling argument, combined with the use of Poincaré inequalities. In a second step, we show that these homogeneous inequalities generalize to the corresponding homogeneous spaces.

\smallskip
 We will consider the following inhomogeneous Sobolev spaces $W^{s,p}\coloneq L^p \cap \dot{W}^{s,p}$, equipped with the norm
     \bes
          \vertii{f}_{W^{s,p}}\coloneq \vertii{f}_{L^p} + \verti{f}_{W^{s,p}}.
          \ees
          
    In the case where $s$ is an integer, $W^{s,p}$ is the classical Sobolev space of $L^p$ functions with all distributional derivatives of order $\leq s$ in $L^p$, and the norm $\vertii{\cdot}_{W^{s,p}}$ is equivalent to 
    \bes
    f \to \vertii{f}_{L^p}+ \sum_{\verti{\alpha}\leq s} \vertii{\partial^{\alpha} f }_{L^p}
    \ees
    (see, e.g.,  \cite[Corollary 12.86]{leoni2024first}).

 \smallskip

   We start by proving Theorem \ref{integer slicing}, using its inhomogeneous counterpart.
\begin{theo}\label{trie}(\cite[Theorem, Section 2.5.6]{triebel2010theory})
 Let $s$ be an integer and  $1<p<\infty$.  There exist $0<K^1_{s,p,N}\leq K^2_{s,p,N}<\infty$ such that, for each $f \in W^{s,p}$, we have
    \begin{flalign*} 
&K^1_{s,p,N}\sum_{i=1}^{N} \left(\int_{\R^{N-1}} \vertii{f(x_1,\dots,x_{i-1},\cdot,x_{i+1},\dots,x_N)}_{W^{s,p}(\R)}^{p} \ d \widehat{x_i} \right)^{1/p} \\ & \hspace{20 pt}\leq \vertii{f}_{W^{s,p}}
    \leq K^2_{s,p,N} \sum_{i=1}^{N} \left(\int_{\R^{N-1}} \vertii{f(x_1,\dots,x_{i-1},\cdot,x_{i+1},\dots,x_N)}^{p}_{W^{s,p}(\R)} \ d \widehat{x_i} \right)^{1/p}. 
    \end{flalign*}    
\end{theo}

\begin{proof}[Proof of Theorem \ref{integer slicing} using Theorem \ref{trie}] 
The left-hand side inequality in \eqref{ttt1} is obvious. We consider the inequality on the right-hand side. We first prove that this inequality holds for $C^\infty_c$ maps. By a scaling argument, it suffices to establish it for  maps supported in $B(0,1)$. 

\smallskip
When $f \in C_c^{\infty}(B(0,1))$, we have  
\bes
\ba
\verti{f}_{W^{s,p}} \leq \vertii{f}_{W^{s,p}} &\leq C \sum_{i=1}^{N} \left(\int_{\R^{N-1}} \vertii{f(x_1,\dots,x_{i-1},\cdot,x_{i+1},\dots,x_N)}^{p}_{W^{s,p}(\R)} \ d \widehat{x_i} \right)^{1/p}
\\ & \leq C \sum_{i=1}^N \left(\int_{\R^{N-1}}\verti{f(x_1,\dots,x_{i-1},\cdot,x_{i+1},\dots,x_N)}^{p}_{W^{s,p}(\R)}   \ d \widehat{x_i}  \right)^{1/p}
\\ & =C \sum_{i=1}^N \vertii{\partial^s_i f}_{L^p}.
\ea
\ees

Here, we rely on Theorem \ref{trie} for the second inequality and on Lemma \ref{1d Poi} for the last one. 

\smallskip
This completes the proof of Theorem \ref{integer slicing} for $f \in C_c^{\infty}$. The fact that Theorem \ref{integer slicing} also holds for each $f \in \dot{W}^{s,p}$ is then a direct consequence of the next result (see \cite[Theorem 11.43]{leoni2024first}).
\end{proof}
\begin{lemma}\label{dens}
    Let $s$ be an integer and $1 \leq p<\infty$. For each $f \in \dot{W}^{s,p}$, there exists $(f_n) \subset C_{c}^{\infty}$ such that $\verti{f_n-f}_{W^{s,p}} \to 0$, as $n \to \infty$.
\end{lemma}
\section{ From inhomogeneous to homogeneous Sobolev embeddings}\label{B}
In this Appendix, we explain how to obtain homogeneous Sobolev embeddings (in homogeneous function spaces) from inhomogeneous Sobolev embeddings (in inhomogeneous function spaces).
Our starting point is the following well-known result (see \cite[Theorem B]{BREZIS20192839}).
   \begin{theo}\label{inhom sob}
    Let $0<s_1<s_2<\infty$ and $1 \leq p_1,p_2 <\infty$ satisfy \eqref{Sob n}. There exists $\tilde{C}\coloneq \tilde{C}_{s_1,s_2,p_1,p_2,N}<\infty$ such that
    \bes
    \vertii{f}_{W^{s_1,p_1}} \leq \tilde{C} \vertii{f}_{W^{s_2,p_2}}, \ \fo f \in W^{s_2,p_2}.
    \ees
\end{theo}

Theorem \ref{inhom sob} clearly implies
\begin{theo}\label{hom sob}
    Let $0<s_1<s_2<\infty$ and $1 \leq p_1,p_2 <\infty$ satisfy \eqref{Sob n}. There exists $\tilde{C}\coloneq \tilde{C}_{s_1,s_2,p_1,p_2,N}<\infty$ such that
    \be\label{compa}
    \verti{f}_{W^{s_1,p_1}} \leq \tilde{C} \verti{f}_{W^{s_2,p_2}}, \ \fo f \in C_c^{\infty}.
    \ee
\end{theo}
\begin{proof}[Proof of Theorem \ref{hom sob} using Theorem \ref{inhom sob}]
    By a scaling argument, it suffices to prove that Theorem \ref{hom sob} holds for smooth functions supported in $B(0,1)$. 

    Let $f \in C^{\infty}(B(0,1))$. We have
    \bes
    \verti{f}_{W^{s_1,p_1}}\leq \vertii{f}_{W^{s_1,p_1}}\leq C \vertii{f}_{W^{s_2,p_2}} \leq C \verti{f}_{W^{s_2,p_2}}.
    \ees
    Here, we rely on Theorem \ref{inhom sob} for the second inequality, and on Lemma \ref{1d Poi} for the last one.
\end{proof}
The proof that estimate \eqref{compa} still holds in $\dot{W}^{s_1,p_1} \cap \dot{W}^{s_2,p_2}$ is more involved. It relies on Lemmas \ref{cauchy}, \ref{dens frac}, and \ref{poincaré} below.

\begin{lemma}\label{cauchy}
    Let $s>0$ and $1 \leq p<\infty$. Let $(f_n) \subset \dot{W}^{s,p}$ be such that $\verti{f_n-f_m}_{W^{s,p}} \to 0$, as $n,m \to \infty$. Then there exists $g \in \dot{W}^{s,p}$ such that $\verti{f_n - g}_{W^{s,p}} \to 0$.
\end{lemma}

The following result is the fractional counterpart of Lemma \ref{dens}.
\begin{lemma}\label{dens frac}
     Let $s$ be non-integer and $1 \leq p<\infty$. For each $f \in \dot{W}^{s,p}$, there exists $(f_n) \subset C_c^{\infty}(\R^N)$ such that $\verti{f_n -f}_{W^{s,p}}\to 0$, as $n \to \infty$.
\end{lemma}

Given $m$ an integer, we denote $\mathscr{P}_m$ the space of polynomials of degree $\leq m$.
We have the following classical result, see, e.g., \cite[Lemma, Section 1.1.11]{maz2013sobolev} and \cite[Theorem 3.5]{devore1984maximal}.
\begin{lemma}\label{poincaré}
Let $R>0$, $s>0$, and $1 \leq p<\infty$. Set $m\coloneq \lfloor s \rfloor$ if $s$ is non-integer, $m \coloneq s-1$ if $s$ is an integer. There exists $C\coloneq C_{s,p,R,N}<\infty$ such that, for each $f \in \dot{W}^{s,p}$, 
\be\label{ineq poinc}
\int_{B(0,R)} \verti{f(x) - P_{f,R}(x)}^p \ dx \leq C\verti{f}_{W^{s,p}}^p,
\ee
for some polynomial $P_{f,R} \in \mathscr{P}_m$.
\end{lemma}

Granted Lemmas \ref{cauchy} and \ref{dens frac}, we turn to
\begin{proof}[Proof of Theorem \ref{Opt Sob} using Theorem \ref{hom sob}]
    Let $f \in \dot{W}^{s_1,p_1} \cap \dot{W}^{s_2,p_2}$. By Lemmas \ref{dens} and \ref{dens frac}, there exists $(f_n) \subset C_c^{\infty}$ such that \be \label{approx}
    \verti{f_n - f}_{W^{s_2,p_2}}\to 0. 
    \ee
    We have, by Theorem \ref{hom sob},
    \bes
     \verti{f_n -f_m}_{W^{s_1,p_1}} \leq \widetilde{C} \verti{f_n -f_m}_{W^{s_2,p_2}}.
    \ees
    Hence, by Lemma \ref{cauchy} there exists $g \in \dot{W}^{s_1,p_1}$ such that $\verti{f_n - g}_{W^{s_1,p_1}}\to 0$. Passing to the limit
yields
\bes
\verti{g}_{W^{s_1,p_1}} \leq \widetilde{C} \verti{f}_{W^{s_2,p_2}}.
\ees

    \smallskip
    We now show that $\verti{g}_{W^{s_1,p_1}}=\verti{f}_{W^{s_1,p_1}}$.  By Lemma \ref{poincaré}, we have, for each $R>0$,
    \bes
     \int_{B(0,R)}\verti{(f_n-g)(x) -P_{f_n-g,R}(x)}^{p_1} \ dx \leq C \verti{f_n-g}_{W^{s_1,p_1}}^{p_1},
    \ees
    where $P_{f_n-g,R}$ is a polynomial of degree $\leq s_1$.
    Hence, \be\label{lo co} (f_n-g) - P_{f_n-g,R} \to 0 \ \text{in} \ L^{p_1}(B(0,R)), \ \fo R>0. \ee

    If $s_2$ is an integer, we argue as follows. By \eqref{lo co}, we have
    \bes
     \partial^{\alpha}\left((f_n-g) - P_{f_n-g,R}\right) = \partial^{\alpha} f_n - \partial^{\alpha} g \to 0, \ \text{in} \ \mathscr{D}'(B(0,R)),
    \ees
    for each $R>0$ and $\alpha$ such that $\verti{\alpha}=s_2$, and thus  $\partial^{\alpha} f_n \to \partial^{\alpha} g $ in $\mathscr{D}'(\R^N)$. On the other hand, for each $\verti{\alpha}=s_2$, $\partial^{\alpha} f_n \to \partial^{\alpha} f$ in $L^p(\R^N)$ (by \eqref{approx}) . Therefore, we have $\partial^{\alpha} g = \partial^{\alpha} f$, for each $\verti{\alpha}=s_2$, and there exists a polynomial $P$ of degree $\leq s_2 -1 $ such that $f-g=P$. But $P \in \dot{W}^{s_1,p}$, since $f$ and $g$ are in $\dot{W}^{s_1,p}$, which implies that $\text{deg}(P)\leq \lfloor s_1 \rfloor $ if $s_1$ is non-integer,  $\text{deg}(P) \leq s_1-1$  if $s_1$ is an integer. This yields $\verti{g}_{W^{s_1,p_1}}=\verti{f-P}_{W^{s_1,p_1}}=\verti{f}_{W^{s_1,p_1}}$ and the desired conclusion. 

    If $s_2$ is non-integer, we argue similarly. For each $h \in \R^N$ and $R>0$, we may find $R'>0$ sufficiently large such that
    \begin{flalign}\label{delta lo}
     &\vertii{\Delta^{\lfloor s_2 \rfloor +1}_h (f_n-g)}_{L^{p_1}(B(0,R))} \\& \hspace{30 pt} =\vertii{\Delta^{\lfloor s_2 \rfloor +1}_h\left( (f_n-g)-P_{f_n-g,R'} \right)}_{L^{p_1}(B(0,R))} \\ & \hspace{30 pt} \leq C \vertii{(f_n-g)-P_{f_n-g,R'}}_{L^{p_1}(B(0,R'))}.
    \end{flalign}
    Combining \eqref{lo co} and \eqref{delta lo}, we find that
    \bes
     \Delta^{\lfloor s_2 \rfloor +1}_hf_n \to \Delta^{\lfloor s_2 \rfloor +1}_h g \ \ \text{in} \ L^{p_1}_{\text{loc}}(\R^N),
    \ees
    for each $h \in \R^N$. By \eqref{approx}, we also have $\Delta^{\lfloor s_2 \rfloor +1}_h f_n \to \Delta^{\lfloor s_2 \rfloor +1}_h f$ in $L^p$, and therefore
    \bes
    \Delta^{\lfloor s_2 \rfloor +1}_h f = \Delta^{\lfloor s_2 \rfloor +1}_h g,
    \ees
    for a.e. $h \in \R^N$. This implies that there exists a polynomial of degree $\leq \lfloor s_2 \rfloor $ such that $f-g=P$. Arguing as in the previous case, we then find that $\verti{f}_{W^{s_1,p_1}}=\verti{g}_{W^{s_1,p_1}}$ and this completes the proof of Theorem \ref{Opt Sob}.
\end{proof}

For the sake of completeness, we now present a possible approach to the proofs of Lemmas \ref{cauchy} and \ref{dens frac}.
For each $s$, we consider the quotient spaces
\bes
\dot{w}^{s,p}\coloneq \begin{cases}
    \dot{W}^{s,p}/\mathscr{P}_{\lfloor s \rfloor }, & \ \text{if} \ s\ \text{is non-integer}, \\
    \dot{W}^{s,p}/\mathscr{P}_{s-1}, & \ \text{if} \ s \ \text{is an integer},
\end{cases}
\ees
equipped with the norms
\bes
\verti{\Bar{f}}_{w^{s,p}}\coloneq  \verti{f}_{W^{s,p}},
\ees
where $\Bar{f}$ is the class of $f$.
We will use results of interpolation theory, see, e.g., \cite[Chapters 16, 17]{leoni2024first}.

For the first result, see \cite[Remark 17.29, Theorem 17.30]{leoni2024first}.
\begin{lemma}\label{interp eq}
    Let $s$ be non-integer and $1 \leq p<\infty$. There exist $ 0<C_1\leq C_2<\infty$ such that
    \bes
     C_1\verti{f}_{W^{s,p}}\leq \vertii{f}_{s/(\lfloor s \rfloor +1),p} \leq C_2 \verti{f}_{W^{s,p}},
    \ees
    for each $f \in L^1_{\text{loc}}$, where 
     $\displaystyle \vertii{\ \cdot \ }_{s/(\lfloor s \rfloor +1),p}$ is the interpolation semi-norm associated to the interpolation space
     $\displaystyle (L^p,\dot{W}^{\lfloor s \rfloor +1,p})_{s/(\lfloor s\rfloor +1),p}$.
\end{lemma}
This result also holds for the quotient spaces $\dot{w}^{s,p}$ and $\dot{w}^{\lfloor s \rfloor +1,p}$:  we have
\be\label{interp}
\dot{w}^{s,p}=(L^p,\dot{w}^{\lfloor s \rfloor, p})_{s/(\lfloor s\rfloor +1),p},
\ee
with equivalence between the interpolation norm and $\verti{ \ \cdot \ }_{w^{s,p}}$.

\begin{proof}[Proof of Lemma \ref{cauchy}]
When $s$ is an integer, $\dot{w}^{s,p}$ is complete (see \cite[Theorem 1, Section 1.1.13]{maz2013sobolev}). Since an interpolation space between Banach spaces is a Banach space (see Theorem \cite[Theorem 16.5]{leoni2024first}), we have the same result if $s$ is non-integer (by \eqref{interp}). This implies Lemma \ref{cauchy}.
\end{proof}

A proof of Lemma \ref{dens frac} using Lemma \ref{interp eq} and interpolation theory may be found in \cite[Proof of Theorem 17.37]{leoni2024first}.

\smallskip
\bibliography{bibart}
\end{document}